\documentclass[]{article}
\usepackage[affil-it]{authblk}
\usepackage{fullpage}
\usepackage{graphicx}
\usepackage{pdflscape}
\usepackage{secdot}
\usepackage{subcaption}
\usepackage{amsmath, amssymb,amsthm}
\usepackage{amsfonts}
\usepackage{latexsym}
\usepackage{booktabs}
\usepackage{enumitem,blindtext}
\usepackage[export]{adjustbox}
\usepackage{mathrsfs}
\usepackage[usenames,dvipsnames]{color}
\usepackage[mathcal]{eucal}
\usepackage{caption}
\usepackage{natbib}
\usepackage{amsmath,amsfonts}
\usepackage[linesnumbered, noend]{algorithm2e}
\usepackage{setspace}
\newcommand{\comments}[1]{}
\DeclareGraphicsExtensions{.png,.pdf,.jpg}
\newtheorem{theorem}{Theorem}

\newcommand{\blue}[1]{{\color{black}#1}} 
\definecolor{purple}{RGB}{153,50,204}

\newcommand{\fac}{I}
\newcommand{\cust}{J}
\newcommand{\MGCLP}{MGCLP}
\newcommand{\JCF}{JCF}
\newcommand{\first}{(F1)}
\newcommand{\second}{(F2)}
\newcommand{\third}{(F3)}
\newcommand{\fourth}{(F4)}

\usepackage[normalem]{ulem}

\usepackage{hyperref}
\begin{document}

\title{An exact solution framework for the multiple gradual cover location problem}

\author[1]{Eduardo \'Alvarez-Miranda \thanks{ealvarez@utalca.cl}}
\author[2]{Markus Sinnl\thanks{markus.sinnl@univie.ac.at}}

\affil[1]{Department of Industrial Engineering, Universidad de Talca, Curic\'o, Chile}

\affil[2]{Department of Statistics and Operations Research, Faculty of Business, Economics and Statistics, University of Vienna, Vienna, Austria}

\date{}
\maketitle

\begin{abstract}
Facility and covering location models are key elements in many decision aid tools in
logistics, supply chain design, telecommunications, public infrastructure planning, and many other industrial and public sectors.
In many applications, it is likely that customers are not dichotomously
covered by facilities, but gradually covered according to, e.g., the distance to the open facilities. Moreover, customers are not served by a single facility,
but by a collection of them, which jointly serve them. In this paper we study the recently introduced
\emph{multiple gradual cover location problem} (\MGCLP).
The \MGCLP\ addresses both of the issues described above.

We provide four different mixed-integer programming formulations for the \MGCLP,
all of them exploiting the submodularity of the objective function and developed a branch-and-cut framework based one these formulations.
The framework is further enhanced by starting and primal heuristics and initialization procedures.

The computational results show that our approach
allows to effectively address different sets of instances. We provide optimal solution values for 13 instances from literature, where the optimal solution was not known, and additionally provide improved solution values for seven instances. Many of these instances can be solved within a minute. We also analyze the dependence of the solution-structure on instance-characteristics.
\end{abstract}

\section{Introduction and motivation}

Facility and covering location models are key elements in many decision aid tools in
logistics, supply chain design, telecommunications, public infrastructure planning, and many other industrial and public sectors.
Classical models typically aim at constructing location decisions that ensure that
all (or as many as possible) customers are covered by as few facilities as possible.
The reader is referred to~\citep[][]{Daskin2013,LaporteEtal2015} for two fundamental textbooks of location theory and science.

As explained in~\cite{berman2018multiple,drezner2014maximin}, the use of a dichotomic scheme for
for modeling coverage (i.e., a given facility either covers a given client or not),
as well as the common consideration that clients are served by a single facility,  
are two aspects that not necessarily respond well to real-world needs. 
Instead, in practical contexts, a facility might cover only a \emph{portion} of the demand of a customer,
and multiple facilities would \emph{jointly} or \emph{cooperatively} serve a customer simultaneously (for instance, each of them would cover a portion of its demand).

The first issue, partial or gradual coverage, has been studied already \blue{since} the 80's (see,~e.g.,~\cite[][]{church1983generalized}),
and the most common way for modeling partial covering is by a distance-based approach.
In these approaches, two critical distances, say $r$ and $R$ ($r \leq R$), are defined.
If the distance between a facility and an allocated customer, say $d$, does not exceed $r$,
then the customer is \emph{fully} covered.
On the contrary, if such distance is greater than $R$,
then the customer is \emph{not} covered at all.
For intermediate distances ($r< d < R$),
the coverage diminishes \emph{gradually} according to some $d$-dependent
function.
For example,~\cite[][]{Thompson1982} proposes, for a super-drugstore chain,
coverage functions  depending on the distance from the facility and city size.
Likewise, in~\cite[][]{JonesSimmons1993} a list of rules
for location of a retail facility are provided, which \blue{is} based
on the use of $r$ and $R$ distances; the authors explain that
such rules appear to be very common in the industry (see~\cite[][]{berman2002generalized}, for further details). 
In this retail industry examples, as well as in any other context,
the coverage level is typically viewed as a
decreasing function of the distance to the facility.
As a matter of fact, there are several alternatives for this function;
relevant examples,
covering the efforts devoted to this topic over the last 25 years, can be found, e.g., in~\cite[][]{BERMAN199841,berman2002generalized,berman2003gradual,berman2009variable,drezner2004gradual,drezner2014maximin,eiselt2009gradual,GhoshEtAl1995} and~\cite[][]{karasakal2004maximal}.

The second issue, \emph{multiple (joint)} coverage, has also been addressed before in the literature.
One of the first works consolidating the concepts on joint and partial coverage,
and its implication in facilities deployment
is presented in~\cite[][]{berman2009cooperative}.
In that paper, the authors explain that multiple location models
have been proposed before in the literature,
and give special emphasis on applications related to the location of siren stations~\cite[][]{CurrentOKelly1992,WeiEtAl2006},
and cell phone towers~\cite[][]{AKELLA2005301}.
It is shown that an adequate model of cooperation ensures
that the coverage of a client is not necessarily performed only by
the closest facility, but rather by a mixed of them based on their proximity.
The authors study the connection \blue{between} joint coverage
and \emph{backup} facilities (see,~e.g.,~\cite[][]{HoganReVelle1986}, for an early reference on this issue).
Based on this connection, one could also link
joint coverage with \emph{reliable} covering models,
which have been proposed when dealing with uncertainty in
the set of available facilities, the set of customers (or, eventually, their demand),
or both. 
A relevant example can be found in~\cite[][]{SnyderDaskin2005},
where the authors proposed a model in which every customer is allocated
to a \emph{main} facility and to a set of \emph{back-up} facilities;
these back-up facilities shall cover the demand of the customer
in case of failure of the main facility (which is likely to be the closest one).
Extensions and variants, following equivalent concepts,
can be found in~\cite[][]{AlbaredaSambola2011,CuiEtAl2010,LiEtAl2013,ShenEtAl2011}.

Although the concepts of gradual coverage and multiple (joint) facility location have been studied already for decades, 
it was only recently that they were \blue{combined} into a single modeling framework aiming at maximizing the total joint cover of all \blue{customers};
this corresponds to the \emph{multiple gradual cover location problem} (\MGCLP), which has been proposed by~\cite[][]{berman2018multiple}.
\blue{The combination is} achieved by the optimization of the \emph{joint coverage function}
introduced in~\citep{drezner2014maximin}, whose details are provided in the following section.
In~\cite[][]{berman2018multiple}, the authors motivate the \MGCLP\ and propose
different (heuristic) algorithmic strategies for solving it (see Section\ref{subsec:jcf} for further details).


\paragraph{Contribution and Outline}

In this paper, we present
four exact mixed-integer programming (MIP) approaches for the \MGCLP. 
These approaches are based on two key elements: 
(i) they exploit the submodularity of the objective function, and
(ii) the use of an exponential number of constraints, which \blue{can be separated efficiently} and are used in a branch-and-cut framework.
We also introduce preprocessing based on domination between facilities and our solution framework also contains starting and primal heuristics. 
Our approaches allow the optimal solution of 13 instances from literature, which have not been solved to optimality before. 
Many of these instances can be solved within a few seconds.

The paper is organized as follows. 
In Section~\ref{sec:form}, we present the formal definition of the
joint coverage function and the \MGCLP, 
and present the four MIP formulations. 
In Section~\ref{sec:bc}, we describe the implementation details of the branch-and-cut algorithms used to solve the corresponding MIP instances;
cut separation, starting and primal heuristics, and preprocessing.
Furthermore, we also present a series of structural properties
of optimal solutions, which enable an effective initialization.
Computational results are given in Section~\ref{sec:compres}.
Finally, concluding remarks are outlined in Section~\ref{sec:con}.

\section{Mixed-Integer Programming formulations for the MGCLP}
\label{sec:form}

\subsection{The joint coverage function}
\label{subsec:jcf}
The joint coverage function, presented in~\citep{drezner2014maximin}, 
and later used in~\cite[][]{berman2018multiple},
is based on probabilistic arguments.
%
Let $\fac$ denote the set of potential facility location and $\cust$ the set of customers. 
Let $0\leq f_{ij}\leq 1$ be the coverage that customer $j \in \cust$ receives from
facility $i \in \fac$ (with $f_{ij}=1$ meaning that the customer gets completely covered, and $f_{ij}=0$ meaning that the customer does not get covered at all by the facility). 
Additionally, let $0\leq \theta \leq 1$ be a given weighting parameter. 
The \emph{joint coverage function (\JCF)} for a customer $j \in \cust $ and a given set of facilities $\fac'$ reads as
\begin{equation}
p_j(\theta, \fac')=\theta \Big( \max_{i \in \fac'} f_{ij}\Big)+(1-\theta) \Big( 1-\prod_{i \in \fac'}(1-f_{ij})\Big) \label{eq:cov} \tag{JCF}
\end{equation}
In this function $f_{ij}$ gets interpreted as the \emph{probability of full coverage}. 
The function combines two extreme cases, namely that coverage events are correlated with a correlation coefficient of one, and that coverage events are independent;
this combination is \emph{controlled} by the parameter $\theta$. 
Note that in contrast to many other location problem, with this function, co-location (i.e., opening more than one facility at a potential location) may be beneficial.

Let $w_j\geq0$ be weights for each customer $j \in \cust$,
and let $K$ be a given (integer) number of facilities to be opened. 
Since, theoretically, we could open the $K$ facilities in a single location,
we define $\fac^K$ as the set of potential facility locations,
having $K$ copies of each facility in $\fac$ (to account for potential co-location).
The \emph{multiple gradual cover location problem (\MGCLP)} is defined as follows \citep{berman2018multiple}. 

\begin{align}
W^*(\theta) = \max_{\fac'\subseteq \fac^K, |\fac'|\leq K} \quad W(\theta,\fac')=\sum_{j \in \cust} w_j p_j(\theta, \fac'), \tag{\MGCLP}
\end{align}
i.e., in the \MGCLP\ we want to open $K$ facilities (potentially more than one at a given location) as to maximize the weighted sum of the \JCF\ over all customer.
As stated above, the \MGCLP\ was introduced in~\cite{berman2018multiple}, 
where a greedy algorithm, an ascent and tabu search heuristic and an approximative MIP-based approach were presented.
\citet{berman2018multiple} show that the greedy algorithm has an approximation guarantee of $1-1/e$ by proving that the objective function $W(\theta,\fac')$ is \emph{nondecreasing} and \emph{submodular} and combining it with the result of
~\cite{nemhauser1978analysis} (see Section~\ref{sec:heur} for details).
The MIP-approach uses the tangent-line approximation (TLA) method for twice-differentiable concave objective functions,
introduced in~\cite{aboolian2007competitive}. 
In this approach, the objective function gets approximated by $L$ line segments. 
In their computational study the authors choose three different $L$ such that the TLA objective is within one, five, and ten percent of optimality, respectively. 

\subsection{MIP models for the MGCLP}
\label{subsec:mipmodels}

In this section, we present four different MIP formulations,
$\first$-$\fourth$, for the \MGCLP, and exploit that the objective function $W(\theta,\fac')$ with $\fac' \in \fac^K$ is a nondecreasing and submodular function. 
Let $\Phi(S)$ be a real valued set-function over the subsets $S$ of a ground-set $N$. Let $\rho_n(S)=\Phi(S \cup \{n\})-\Phi(S)$ for all $S \subset N$ and $n \in N$, i.e., the marginal gain achieved by adding element $n$ to the set $S$. Such a function $\Phi(S)$ is a nondecreasing and submodular function, iff
$\Phi(T)\leq\Phi(S) + \sum_{n \in T \setminus S}\rho_n(S), \forall S,T \subseteq N$
(see, e.g., ~\cite{nemhauser1978analysis}, which also presents additional equivalent definitions)

The four formulations built-on each other and are based on~\cite{nemhauser1981maximizing}, 
where it is shown that for any nondecreasing submodular function maximization problems can be formulated as MIPs by introducing an additional (continuous) variable $\eta$.
This variable is used to measure the value of the objective function. 
The correctness of the objective value is ensured by an exponential family
of cuts on $\eta$, as encoded by expression~\eqref{eq:scuts}. 
Let $\mathbf{z}\in\{0,1\}^{|N|}$ denote the characteristic vector of the set $N$. 
The family of cuts is given by
\begin{equation}
\eta \leq \Phi(S) + \sum_{i \in N \setminus S}\rho_i(S)z_i, \quad \forall S \subseteq N \label{eq:scuts} \tag{SCuts}
\end{equation}

Note that while all of these cuts are valid, for correctness for problems with a cardinality constraint of value $K$ on the size of $S$ (like the \MGCLP), it is enough to add all cuts for $|S|=K$. As shown in~\citep{berman2018multiple}, the objective function $W(\theta,\fac')$ of the \MGCLP\ is nondecreasing and submodular.
In particular, it is the (weighted) sum of the nondecreasing and submodular functions $p_j(\theta,\fac')$, 
which in turn are a convex combination of two nondecreasing and submodular functions,
namely $\max_{i \in \fac'} f_{ij}$ and $1-\prod_{i \in \fac'}(1-f_{ij})$.

Let $\mathbf{x}\in\{0,1\}^{|\fac|\times K}$ be a binary vector such that $x_i^k=1$, if the $k$-th facility ($1\leq k \leq K$) at location $i \in \fac$ is opened (recall that co-location is possible);
and  $x_i^k=0$, otherwise.
When needed, we write $i^k$ to refer to the $k$-th facility at location $i$ and identify an element of $I^K$ by $(i,k)$.
The following constraint~\eqref{eq:card} ensures, that at most $K$ facilities are opened.
\begin{equation}
\sum_{i \in I, 1 \leq k \leq K} x_i^k \leq K. \label{eq:card} \tag{CARD}
\end{equation}

Additionally, in order to exclude symmetric solutions with regard to co-location, 
the following constraints can be used
\begin{equation}
x_i^k \leq x_i^{k+1},\quad \forall i \in \fac, 1\leq k \leq K-1; \label{eq:sym} \tag{SYM}
\end{equation}
these constraint ensure that the $(k+1)$-th facility in a location is only opened, 
if the $k$ facilities with lower $k$ have been opened before. 
These constraints are not necessary for correctness of the formulations, 
but proved very helpful in preliminary computations and thus are used in all formulations.

For a given set of facilities $\fac' \subset \fac^K$ and a given vector $\mathbf{x}$,
let $\phi(\mathbf{x},\fac')$ denote the right-hand-side of cuts~\eqref{eq:scuts} 
when applied to the objective function $W(\theta,\fac')$ of the \MGCLP. 
Note that the coefficients $\rho_i$ for $(i,k) \in \fac^K \setminus \fac'$ of a cut can be easily calculated by calculating $p_j(\theta,\fac' \cup \{i\})-p_j(\theta,\fac')$ for each customer $j \in \cust$ and summing up, i.e., we obtain
\begin{equation}
\phi(\mathbf{x},\fac')=W(\theta,\fac')+\sum_{(i,k) \in \fac^K \setminus \fac} \Big(\sum_{j \in \cust}(p_j(\theta,\fac' \cup \{i\})-p_j(\theta,\fac'))\Big) x^k_i.
\end{equation}

Using this notation,~\eqref{eq:obj1}-\eqref{eq:vardef1} gives a first formulation $\first$;
\begin{align} 
\first \quad W^*(\theta) &= \max\quad \eta  \label{eq:obj1} \tag{F1.1} \\
\eta &\leq \phi (\mathbf{x}, \fac'),\; \forall \fac' \in \fac^K: |\fac'|=K \label{eq:scuts1} \tag{F1.2} \\ 
&\eqref{eq:card},\;\eqref{eq:sym}\;\mbox{and}\; \mathbf{x}\in\{0,1\}^{|\fac|\times K}.\label{eq:vardef1} \tag{F1.3}
\end{align} 
The objective function~\eqref{eq:obj1} and constraints~\eqref{eq:scuts1} ensure that the objective is correct. 
As there is an exponential number of constraints~\eqref{eq:scuts1},
our strategy for tackling this formulation relies on
adding them on-the-fly when they are violated, within a branch-and-cut scheme
(the separation of the constraints is discussed in Section~\ref{sec:sep}). 

The second formulation, $\second$, encoded by constraints~\eqref{eq:obj2}-\eqref{eq:vardef2},
is given by
\begin{align}
\second  \quad W^*(\theta) &=  \max\quad \sum_{j \in \cust} \eta_j \label{eq:obj2} \tag{F2.1} \\
\eta_j &\leq \phi_j (\mathbf{x}, \fac'),\; \forall j \in \cust, \forall \fac' \in \fac^K: |\fac'|=K  \label{eq:scuts2} \tag{F2.2} \\ 
&\eqref{eq:card},\;\eqref{eq:sym},\;\mbox{and}\;
  \mathbf{x}\in\{0,1\}^{|\fac|\times K} \label{eq:vardef2}. \tag{F2.3}
\end{align}
In this second formulation, we exploit the fact that the objective function of the \MGCLP\ \emph{decomposes} by customer 
(as it it the sum of the functions $p_j(\theta,\fac')$ for each customer). 
Thus, instead of a single-variable $\eta$ to measure the objective, 
we use continuous variables $\eta_j$ for each $j \in \cust$, 
and have the sum of these variables in the objective function. 
Hence, each variable $\eta_j$ now has an individual family of cuts $\phi_j (\mathbf{x}, \fac')$ of type~\eqref{eq:scuts}, 
which ensures that the part in the objective contributed by the coverage of customer $j$ is correct.

In the third formulation $\third$, 
we further exploit the decomposability of the objective function $W(\theta,I')$. 
We model the $\max_{i \in \fac'} f_{ij}$ component of each $p_j(\theta,\fac')$ by using additional variables
and only use cuts to deal with the $1-\prod_{i \in \fac'}(1-f_{ij})$ component.
Note that in the MIP model of~\cite{berman2018multiple},
the $\max$-part was formulated in a similar way, while the second part was approximated using the TLA approach. 
Let $\mathbf{y}\in\{0,1\}^{|\fac|\times |\cust|}$ be a vector of binary
variables such that $y_{ij}=1$ if the maximum of $f_{ij}$ for a $j \in \cust$ is obtained by opening facility $i$, and $y_{ij}=0$, otherwise.
Let $\phi^P_j (\mathbf{x}, \fac')$ denote the cuts of type~\eqref{eq:scuts} for $j \in \cust$ associated with the $1-\prod_{i \in \fac'}(1-f_{ij})$-part of the objective. 
We obtain the following formulation $\third$;
\begin{align}
\third  \quad W^*(\theta) &= \max\quad \sum_{i \in \fac} \sum_{j \in \cust} \theta w_j f_{ij} y_{ij} + \sum_{j \in \cust} \eta_j \label{eq:obj3}\tag{F3.1} \\
\eta_j &\leq \phi^P_j (\mathbf{x}, \fac'),\; \forall j \in \cust, \forall \fac' \in \fac^K: |\fac'|=K  \label{eq:scuts3}  \tag{F3.2}\\ 
\sum_{i \in \fac} y_{ij}&\leq 1 ,\; \quad \forall j \in \cust \label{eq:max31}\tag{F3.3} \\
y_{ij}&\leq x_i^1,\; \quad \forall i \in \fac, \forall j \in \cust \label{eq:max32} \tag{F3.4}\\
\eqref{eq:sym},&\;\eqref{eq:card},\;\mathbf{x}\in\{0,1\}^{|N|\times K}\;\mbox{and}\;\mathbf{y}\in\{0,1\}^{|\fac|\times |\cust|} \label{eq:vardef3x} \tag{F3.5}
\end{align}
In this MIP model, cuts~\eqref{eq:scuts3} ensure that the product-part of the objective function is correctly measured. 
Constraints~\eqref{eq:max31} make sure that only one facility can contribute to the max-part of any customer. 
Moreover, constraints~\eqref{eq:max32} model the fact that, if a facility at some location wants to contribute to the max-part, it must be opened;
in particular, the first facility (of the $K$ copies) at this location must be opened. 
Hence, these constraints complement the ordering imposed by the symmetry constraints~\eqref{eq:sym}.

Finally, in formulation $\fourth$, we again exploit the decomposability of the objective function.
Compared to $\third$, we also model the max-part of the objective by using cuts, 
i.e., both parts of the objective are now modeled using cuts. 
To this end, for each customer $j \in \cust$, we introduce two continuous variables,
$\eta^M_j$ and $\eta^P_j$,
to measure the contribution of the max-part and product-part, respectively. 
The cuts for the max-part actually have a nice form and are of polynomial size (see~\cite{nemhauser1981maximizing} for further details):
Assume that for a given customer $j$, 
the $f_{ij}$ values are ordered in a nondecreasing order, i.e.,
$f_{|\fac|j}\geq f_{(|\fac|-1)j} \geq \ldots \geq f_{1j} \geq f_{0j}$, 
with $f_{|\fac|-0j}$ defined to be zero. 
Let $(\cdot)^+=\max\{0,\cdot\}$. 
Then these cuts are of the form
\begin{equation*}
 \eta^M_j \leq \theta w_j \Big(f_{rj} +\sum_{i \in \fac} (f_{ij}-f_{rj})^+ x^1_i \Big), \quad r=0, \ldots, |\fac|-1.
 \end{equation*}
Note that only the first facilities for each location are involved in the cuts. 
Although they are of polynomial size, there can still be many of these cuts;
hence, their separation is also embedded within a branch-and-cut fashion. Let $\phi^M_j (\mathbf{x}, r)$ denote these cuts for $r=0, \ldots, |\fac|-1$.
Using these cuts, we obtain the formulation \fourth;
\begin{align}
\fourth \quad W^*(\theta) &=\max\quad \sum_{j \in \cust} (\eta^M_j + \eta^P_j)  \label{eq:obj4}\tag{F4.1} \\
\eta^P_j &\leq \phi^P_j (\mathbf{x}, \fac'),\; \forall j \in \cust, \forall \fac' \in \fac^K: |\fac'|=K  \label{eq:scuts4a}\tag{F4.2} \\ 
\eta^M_j &\leq \phi^M_j (\mathbf{x}, r),\; \forall j \in \cust, r=0, \ldots, |\fac|-1 \label{eq:scuts4b} \tag{F4.3}\\ 
&\eqref{eq:sym},\;\eqref{eq:card}\;\mbox{and}\;\mathbf{x}\in\{0,1\}^{|N|\times K}. \label{eq:vardef4}\tag{F4.4}
\end{align}


\section{Implementation details of the branch-and-cut algorithms \label{sec:bc}}

All the four formulation have an exponential number of constraints
(\eqref{eq:scuts1},~\eqref{eq:scuts2},~\eqref{eq:scuts3} and~\eqref{eq:scuts4a}),
for ensuring the correctness of the objective function.
These cuts (and also the polynomial-sized family~\eqref{eq:scuts4b}) are separated on-the-fly using branch-and-cut approaches. 
In this section, separation of the cuts is described, 
as well as further ingredients of the branch-and-cut approaches.

\subsection{Separation of cuts \label{sec:sep}}

The separation of cuts~\eqref{eq:scuts1}, is performed as follows.
Let $(\tilde{\mathbf{x}}, \tilde \eta)$ be the values of the LP-relaxation at a
given branch-and-bound node. 
If $\tilde{\mathbf{x}}$ is binary, then an exact separation of the cuts can be done by calculating $W(\theta,\bar \fac)$, where $\bar \fac=\{i^k:\tilde x^k_i=1\}$,
i.e., the open facilities induced by $\tilde{\mathbf{x}}$. 
If $W(\theta,\bar \fac)>\tilde \eta$, the current LP-solution violates~\eqref{eq:scuts1}, and we add the cut induced by $\bar \fac$. 

Note that the exact separation of cuts for the case of binary $\tilde{\mathbf{x}}$ is enough to ensure correctness of our approach. 
Nonetheless, we also implemented a heuristic separation for the case of $\tilde{\mathbf{x}}$ being fractional. 
In this case, we sort the facilities $i^k$ non-increasingly according to the values of $\tilde x^k_i$ and construct $\bar \fac$ by taking the $K$ first facilities of the sorting. 
This solution is then used to induce the corresponding cut
(which is added, if violated).
 
The separation of the other cuts,~\eqref{eq:scuts2},~\eqref{eq:scuts3},~\eqref{eq:scuts4a} and~\eqref{eq:scuts4b},
is performed in an equivalent manner.

\subsection{Starting heuristic and primal heuristic \label{sec:heur}}

To initialize the branch-and-cut framework, we add a feasible starting solution. 
This starting solution is constructed using the greedy $(1-1/e)$ approximation algorithm of~\cite[][]{nemhauser1978analysis}, 
which is also used in~\cite[][]{berman2018multiple} as one of the tested approaches.
We also try to improve the solution constructed by the greedy algorithm with a local search. 
In the local search phase, for each opened facility in the solution,
we try to replace it with another one, if it gives an improved solution value. 
We pass through all opened facilities 
(starting with the one added last by the greedy algorithm) 
and repeat this procedure, until no improvement is found in a round of iterations. 
The starting heuristic is outlined in Algorithm~\ref{alg:greedy}.

\begin{algorithm}[h!tb]   
\DontPrintSemicolon                 
\SetKwInOut{Input}{input}\SetKwInOut{Output}{output}
\Input{instance $(\fac,\cust, f, \theta, K)$ of the \MGCLP }
\Output{feasible solution $S$}
$S\leftarrow \emptyset$\;
\For{$k=1$ \KwTo $K$ }{
$i^*=\arg \max_{i \in \fac} W(\theta, S \cup \{i\})$ \label{alg:argmax}\;
$S\leftarrow S \cup \{i^*\}$\;
}
\texttt{improve} $\leftarrow true$\;
\While{improve}
{
\texttt{improve} $\leftarrow false$\;
\For{$k=K$ \KwTo $1$ }{
$S'\leftarrow S \setminus S[k]$\;
\For{$i \in \fac$}{
\If{$W(\theta,S' \cup \{i\})>W(\theta,S)$}
{
$S\leftarrow S' \cup \{i\}$\;
\texttt{improve} $\leftarrow true$\;
break\;
}
}
}
}
\caption{Greedy Heuristic of~\cite[][]{nemhauser1978analysis} applied to the \MGCLP, complemented by a local search phase. \label{alg:greedy} }
\end{algorithm}

To speed-up the evaluation of $\arg \max_{i \in \fac} W(\theta, S \cup \{i\})$ in line~\ref{alg:argmax}, 
we use \emph{lazy evaluation} (see~\cite{leskovec2007cost}), 
which exploits the submodularity of $W$, according to the following rule.
In the first iteration, we calculate $W(\theta,\{i\})$ for each $i \in \fac$. 
We add $i^*$ to $S$ (due to the $\arg \max$ criterion),
and also store all the facilities in a priority queue, sorted by decreasing values of $W(\theta,\{i\})$. 
In the remaining iterations, instead of calculating $W(\theta,S \cup \{i\})$ 
for each $i \in \fac$, we start by calculating $\rho'=W(\theta,S \cup \{i'\})$,
for the top-element $i'$ in the priority queue. 
We then compare the value $\rho'$ with the stored value $\rho''$ 
(of the second element $i''$ in the priority queue). 
If $\rho' \geq \rho''$, we have that $i'$ gives the $\arg \max$, 
since due to submodularity, the values of $W(\theta,S \cup \{i\})$ are non-increasing when the size of $S$ increases. 
If $\rho'<\rho''$, we re-insert $i'$ in the priority queue with value $\rho'$,
and repeat the procedure for $i''$.

During the branch-and-cut, 
we use a modified version of Algorithm~\ref{alg:greedy} as primal heuristic. 
In this modified version, in line~\ref{alg:argmax}, 
we use $\arg \max_{i \in \fac} \tilde x^k_i W(\theta, S \cup \{i^k\})$,
with $\tilde{\mathbf{x}}$ being a (eventually integer) solution
at a given node of the branch-and-bound tree.
In order to save computation time on unnecessary runs of the local search, 
we store all solution values of (intermediate) solutions found during previous runs of the heuristic in a hash-map, and stop the run of the local search, 
if the currently constructed solution has the same value as a previously encountered solution.

\subsection{Preprocessing and initialization}

Due to co-location, the number of variables in the models can be very large, 
as there could be instances where, e.g., 
the optimal solution only uses one location, and all $K$ facilities are opened there. In this section, we give results that allow removing of some of the potential co-locations by showing that they will never be used in an optimal solution (in our framework, we set all variables associated which such co-locations to zero).

%

For the results presented below, 
we assume that the ordering constraint~\eqref{eq:sym} is part of the corresponding MIP model; hence, given a location $i$ and positions $l,k$ with $k<l$ 
a facility $i^k$ will always be opened \emph{before} a facility with higher superscript $i^l$.
The first result is given by the following theorem.
\begin{theorem}
Let $i \in \fac$ be a facility location with each $f_{ij}=1$ or $f_{ij}=0$, $j \in \cust$. Then in an optimal solution, at most one facility will be opened at location $i$.
\end{theorem}
\begin{proof}
Opening another facility at this location will neither improve the max-part of the objective function, nor the product-part.
\end{proof}

Complementary, the next result exploits that the objective function $W(\theta,\fac')$ is a nondecreasing and submodular function. 
For a facility $i \in \fac$ an and integer $k$, 
let $\mathcal I^k=\{i^{k'}:k'\leq k\} $, i.e., 
be the set of the first $k$ facilities at this location, with $\mathcal I^0=\emptyset$.

\begin{theorem}
Let $UB^{K-k}$ denote an upper bound for the objective value for any solution with $K-k$ open facilities and $z$ be the value of a feasible solution. 
If $W(\theta,\mathcal I^k)+UB^{K-k}<z$, then no facilities will be opened in positions $k$ to $K$ at location $i$ in any optimal solution.
\end{theorem}

\begin{proof}
Since $W(\theta,\cdot)$ is a nondecreasing and submodular function we have that $\sum_{k'\leq k} \Big(\sum_{j \in \cust}(p_j(\theta,\mathcal I^k)-p_j(\theta,\mathcal I^{k-1}))\Big)=W(\theta,\mathcal I^k)$ is an upper bound on the marginal gain, 
which can be achieved by opening the first $k$ facilities at location $i$. 
Thus, taking any solution where $K-k$ facilities have been opened and then opening the first $k$ facilities at $i$, 
we can never get a solution with objective at least $z$, and hence, such a solution cannot be optimal.
\end{proof}

To use this results, an upper bound $UB^{K-k}$ on the objective value for any solution with $k'=K-k$ open facilities is needed. 
We use two different ways to calculate such a bound, and then use the smaller value.
The first way is to use the greedy $1-(1/e)$ approximation algorithm 
(see Section~\ref{sec:heur}). 
Let $z^{k'}$ be the value of the solution constructed at step $k'$, 
as the algorithm has an $1-(1/e)$ approximation guarantee, 
we have that  $UB^{k'} \leq \frac{z^{k'}}{1-(1/e)}$. 
The second way consists of calculating the marginal gain $\sum_{j \in \cust}(p_j(\theta,\mathcal I^l)-p_j(\theta,\mathcal I^{l-1}))$ 
for all $i$ and $l$ up to $k'$, sorting the resulting values in a nonincreasing way,
and then summing up the first $k'$ values.

In addition to above results, which allow the removing of co-locations, 
there is also the following dominance result, 
which allows the removal of facility locations. 

\begin{theorem}
Let $i,i' \in \fac$ be two facilities, with $f_{ij}\geq f_{i'j}$ for each $j \in \cust$. Then, in an optimal solution, no facility will be opened at location $i'$.
\end{theorem}
\begin{proof}
Suppose there is a solution $S'$ where a facility is opened at location $i'$.
Construct another solution $S$, where the open facility at $i'$ is replaced with an open facility at location $i$. 
Since $f_{ij}\geq f_{i'j}$ for each $j \in \cust$, 
the objective value of $S$ is at least as large as the one of $S'$.
\end{proof}


Moreover, we also add all cuts induced by $I'=\emptyset$ to initialize our framework. 
The efficacy of this initialization is evaluated
in the following section,


\section{Computational results \label{sec:compres}}

The branch-and-cut framework was implemented in C++ using CPLEX 12.7, 
which was left at default settings. 
The runs were carried out on an Intel Xeon E5 v4 CPU with 2.2 GHz and 6GB memory and using a single thread. 
The timelimit for a run was set to 600 seconds.

\def\seta{\texttt{pm-$5$-$20$-$0.2$}}
\def\setb{\texttt{pm-$5$-$20$-$0.5$}}
\def\setc{\texttt{pm-$5$-$20$-$0.8$}}
\def\setd{\texttt{pm-$10$-$25$-$0.2$}}
\def\sete{\texttt{pm-$10$-$25$-$0.5$}}
\def\setf{\texttt{pm-$10$-$25$-$0.8$}}

\subsection{Instance description}
\label{subsec:instdesc}

To evaluate the effectiveness of our approach, 
we used the same instances as in~\cite{berman2018multiple}. 
They are based on 40 p-median instances from the OR-library~\cite{ORLibrary}. 
Each node of the instances is a customer and also a potential facility location, and the weights are uniform. 
The instances have up to 900 nodes and $K$ is up to 200 
(see Tables~\ref{ta:ta11}-\ref{ta:ta23} for the values for each instance). 
To define the coverage rates $f_{ij}$, $i \in \fac, j \in \cust$ a linear decline function is used to convert the distances $d_{ij}$ of an instance to values $f_{ij}$. In particular, given two threshold values $r\geq 0$ and $R>r$, the values $f_{ij}$ are defined as follows:
\begin{equation*}
f_{ij}= \begin{cases}
1 &\mbox{if } d_{ij} \leq r \\
1-\frac{d_{ij}-r}{R-r} & \mbox{if } r < d_{ij} < R \\
0 &\mbox{if } d_{ij} \geq R  \\
 \end{cases}.
\end{equation*}
In~\cite{berman2018multiple}, the authors use $r=5$ and $R=20$, along with $\theta=0.2$. 
In addition to this, we also tested on instances using $r=10$ and $R=25$ 
(thus, having a larger number of $f_{ij}>0$) and also $\theta \in \{0.2, 0.5, 0.8\}$. 
In total, this gives 240 instances.
We refer to the resulting instance set as \texttt{pm-$r$-$R$-$\theta$}, e.g., 
\seta\ are the instances used in~\cite{berman2018multiple}. 

\def\basic{\texttt{b}}
\def\frac{\texttt{f}}
\def\heur{\texttt{fh}}
\def\all{\texttt{fhp}}

\subsection{Assessing the effectiveness of the proposed strategies}

First, we give an overview on the performance of the different formulations and also of the different ingredients of our framework 
(i.e., separation on fractional solutions, preprocessing, initialization, and heuristics). 
In particular, for each of the formulation, we tested the following four configurations:

\begin{itemize}
\item \basic: In this basic setting, we only use the separation for integer solutions, and do not use the preprocessing, initialization, nor the (starting and primal) heuristic.

\item \frac: In this setting, we also do the separation for fractional solutions.

\item \heur: This is setting \frac\ together with the starting heuristic and the primal heuristic.

\item \all: This is setting \heur\ together with the preprocessing and also the initialization, i.e., all ingredients of our framework are turned on.

\end{itemize}

The computational study was carried out on all instances with up to 400 nodes 
(these are 120 instances). 
In Figures~\ref{fig:run1}-\ref{fig:run4} we report the performance profile plots of the runtime to optimality, while in Figures~\ref{fig:gap1}-\ref{fig:gap4}
we report the performance profile plots of the attained optimality gap $g [\%]$ (calculated as $100 \cdot (UB-z^*)/(z^*)$, where $UB$ is the upper bound and $z^*$ is the value of the best solution found) for all formulations, instances and settings.


\begin{figure}
\begin{subfigure}[b]{.5\linewidth}
\centering \includegraphics[width=.99\linewidth]{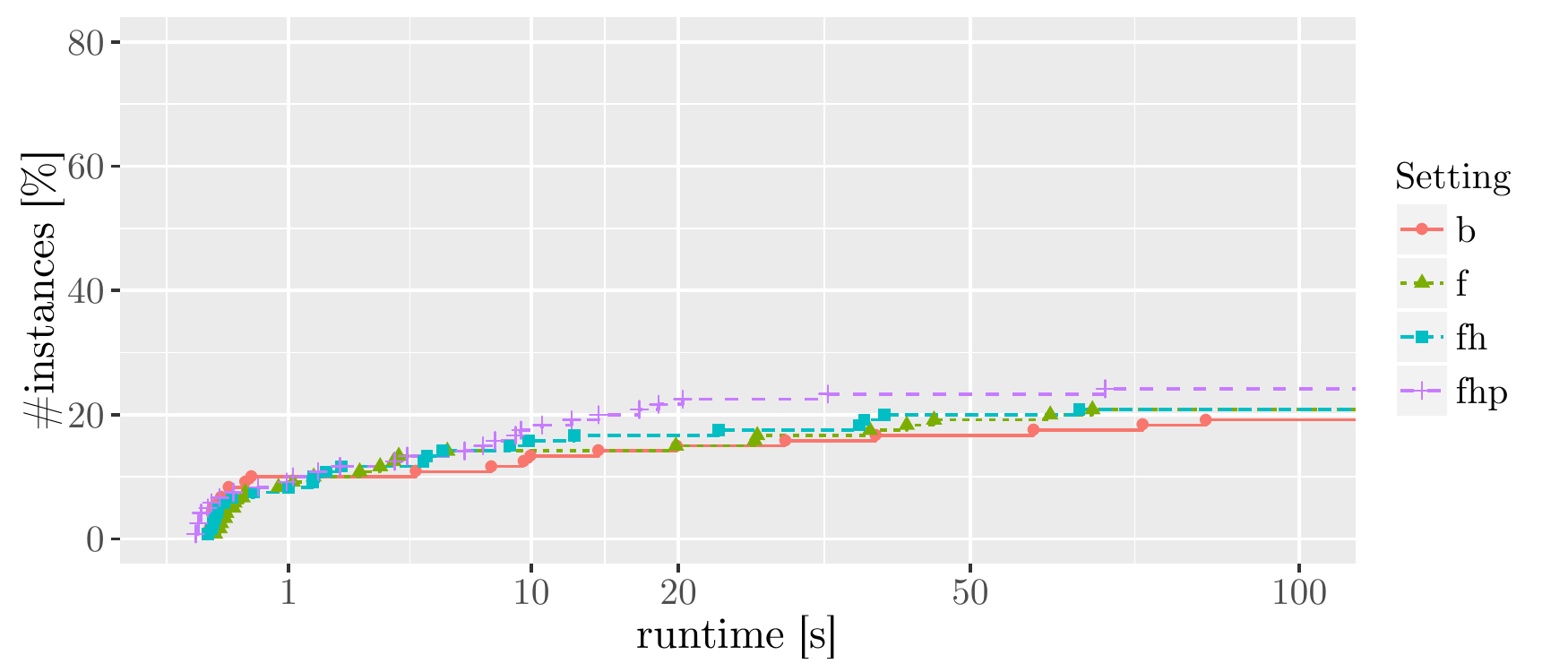}
\caption{Runtime performance of \first}\label{fig:run1}
\end{subfigure}%
\begin{subfigure}[b]{.5\linewidth}
\centering
\centering \includegraphics[width=.99\linewidth]{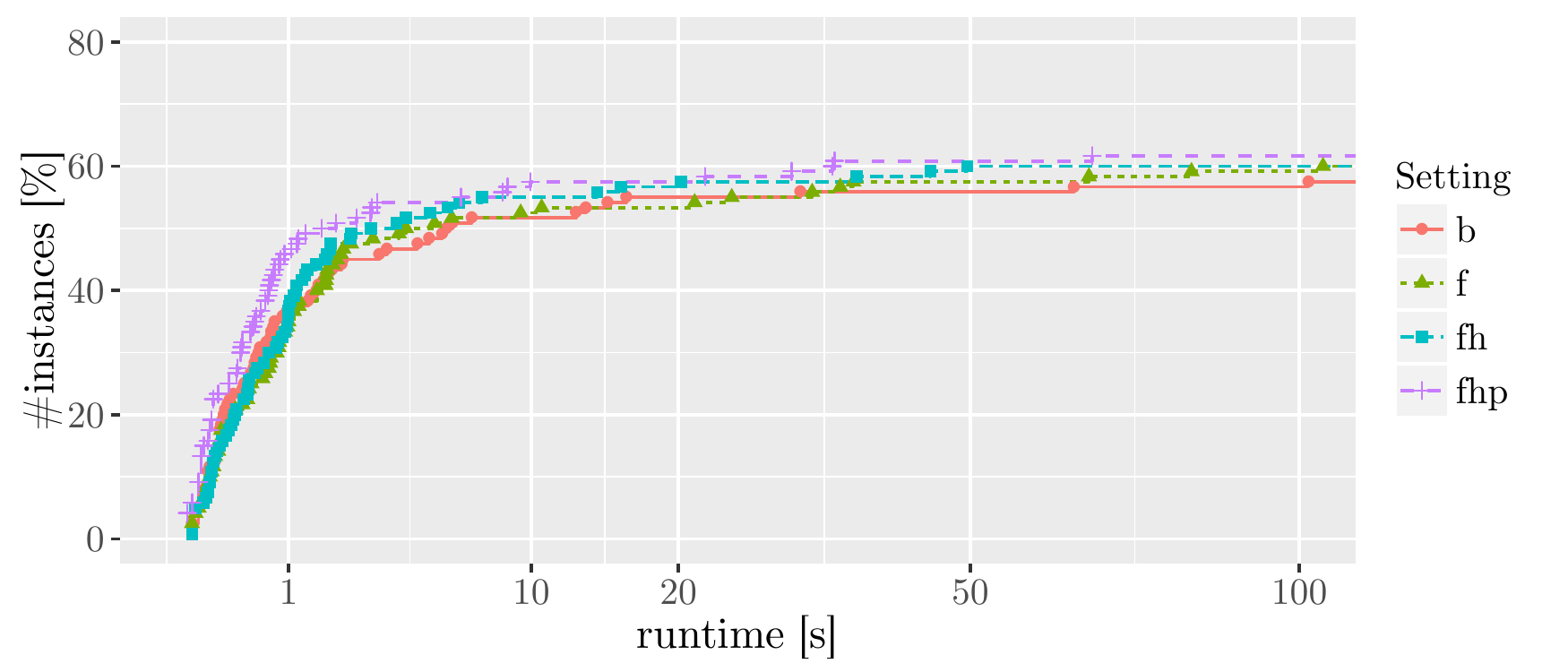}
\caption{Runtime performance of \second}\label{fig:run2}
\end{subfigure}
\begin{subfigure}[b]{.5\linewidth}
\centering
\includegraphics[width=.99\linewidth]{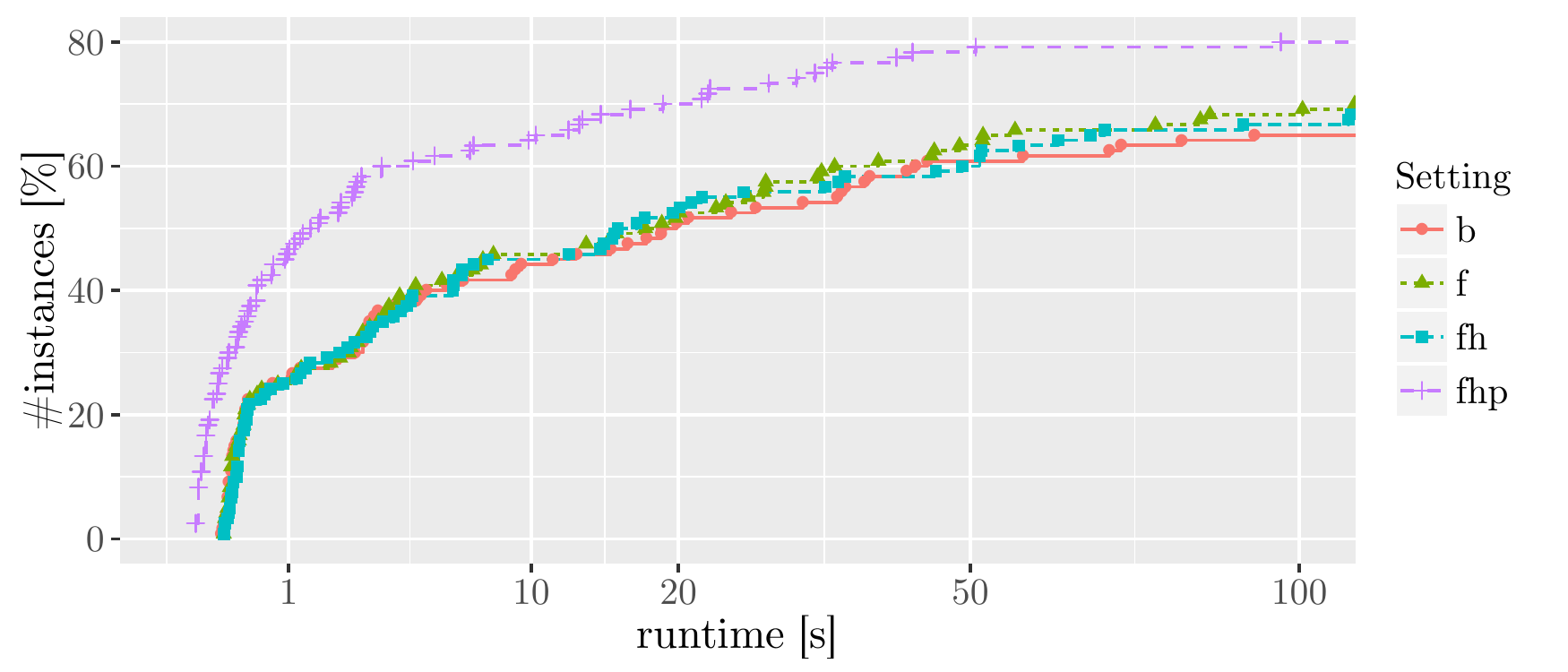}
\caption{Runtime performance of \third}\label{fig:run3}
\end{subfigure}%
\begin{subfigure}[b]{.5\linewidth}
\centering
\includegraphics[width=.99\linewidth]{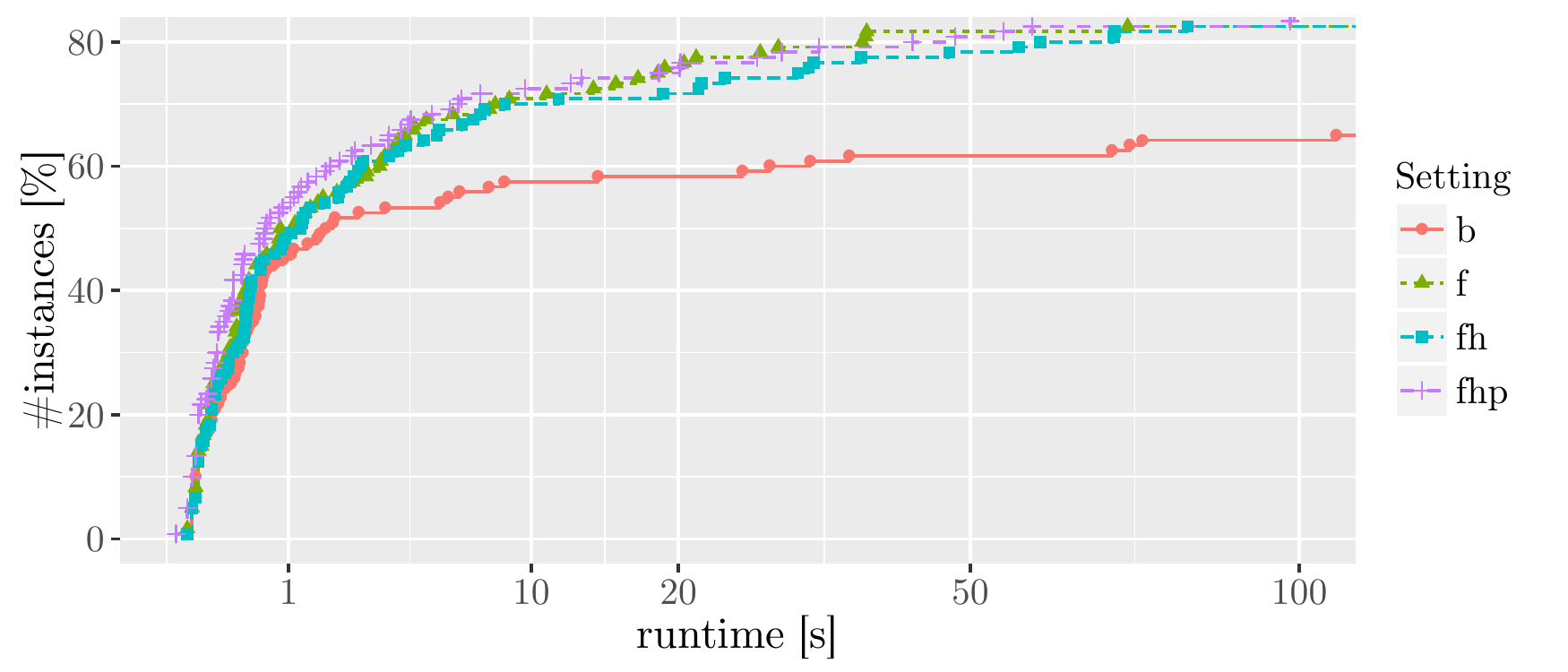}
\caption{Runtime performance of \fourth}\label{fig:run4}
\end{subfigure}
\caption{Performance profiles of runtimes for~\first-\fourth\ and different settings}\label{fig:run}
\end{figure}

\begin{figure}
\begin{subfigure}[b]{.5\linewidth}
\centering \includegraphics[width=.99\linewidth]{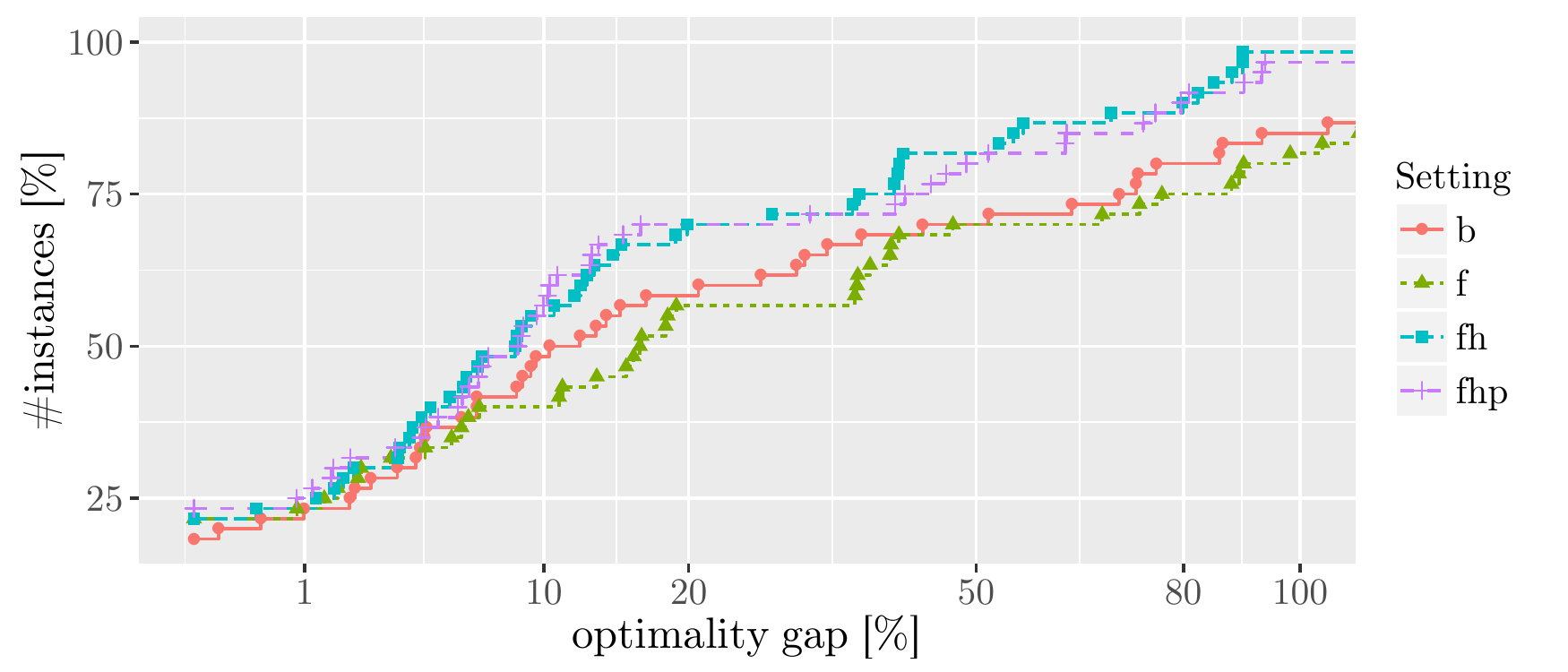}
\caption{Optimality gap performance of \first}\label{fig:gap1}
\end{subfigure}%
\begin{subfigure}[b]{.5\linewidth}
\centering
\centering \includegraphics[width=.99\linewidth]{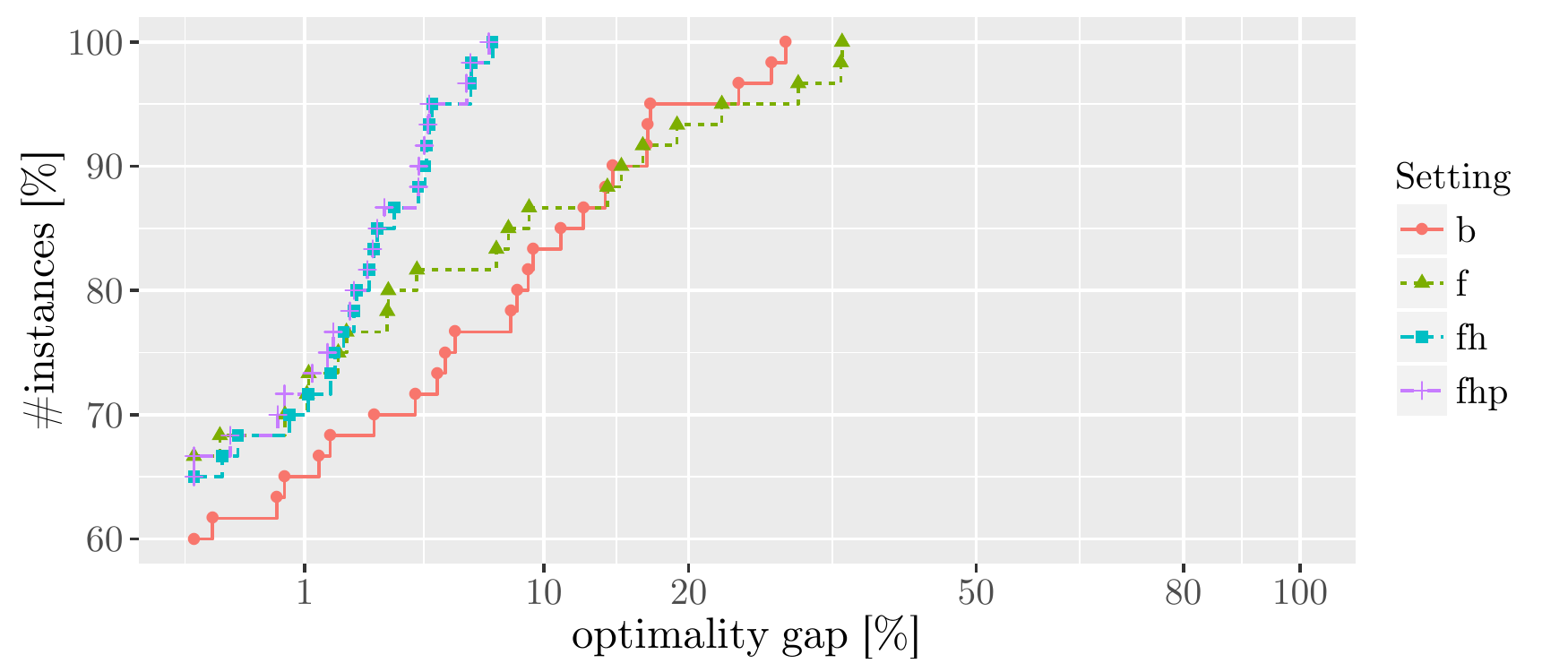}
\caption{Optimality gap performance of \second}\label{fig:gap2}
\end{subfigure}
\begin{subfigure}[b]{.5\linewidth}
\centering
\includegraphics[width=.99\linewidth]{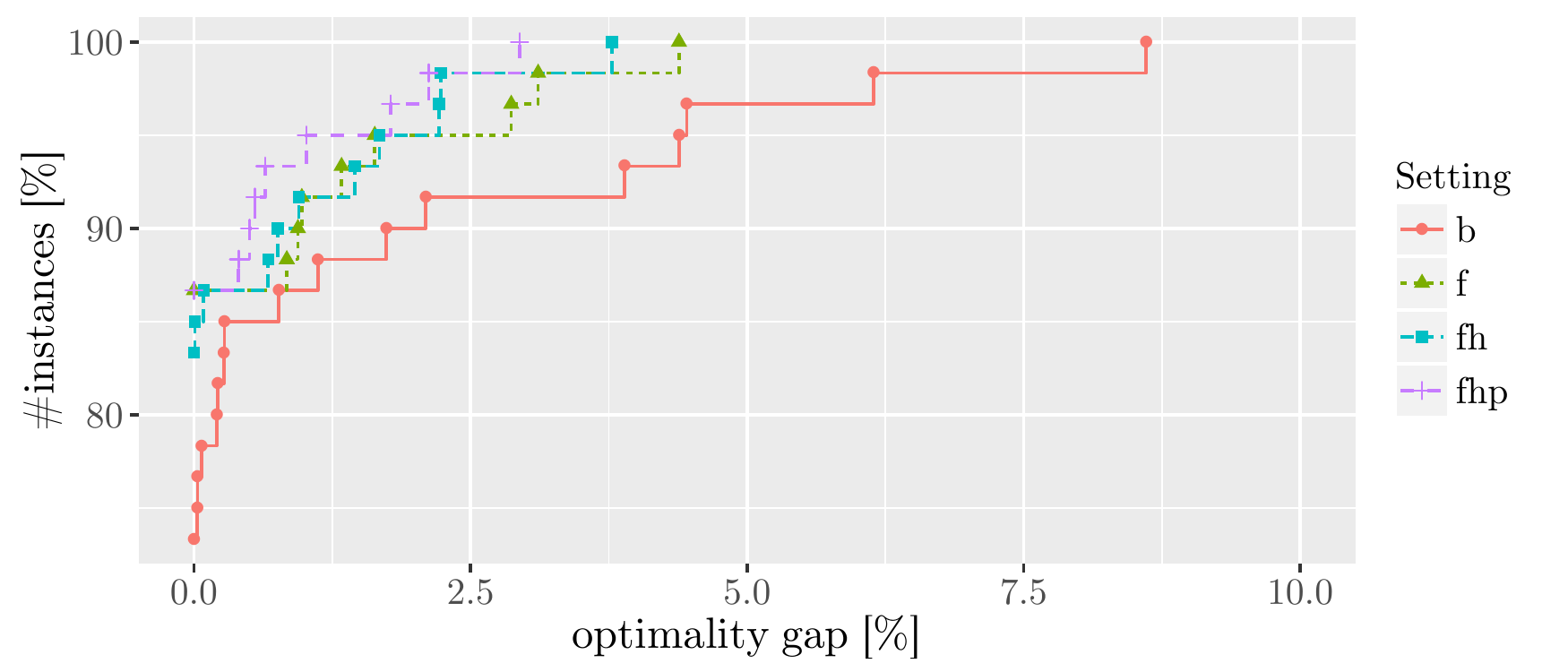}
\caption{Optimality gap performance of \third}\label{fig:gap3}
\end{subfigure}%
\begin{subfigure}[b]{.5\linewidth}
\centering
\includegraphics[width=.99\linewidth]{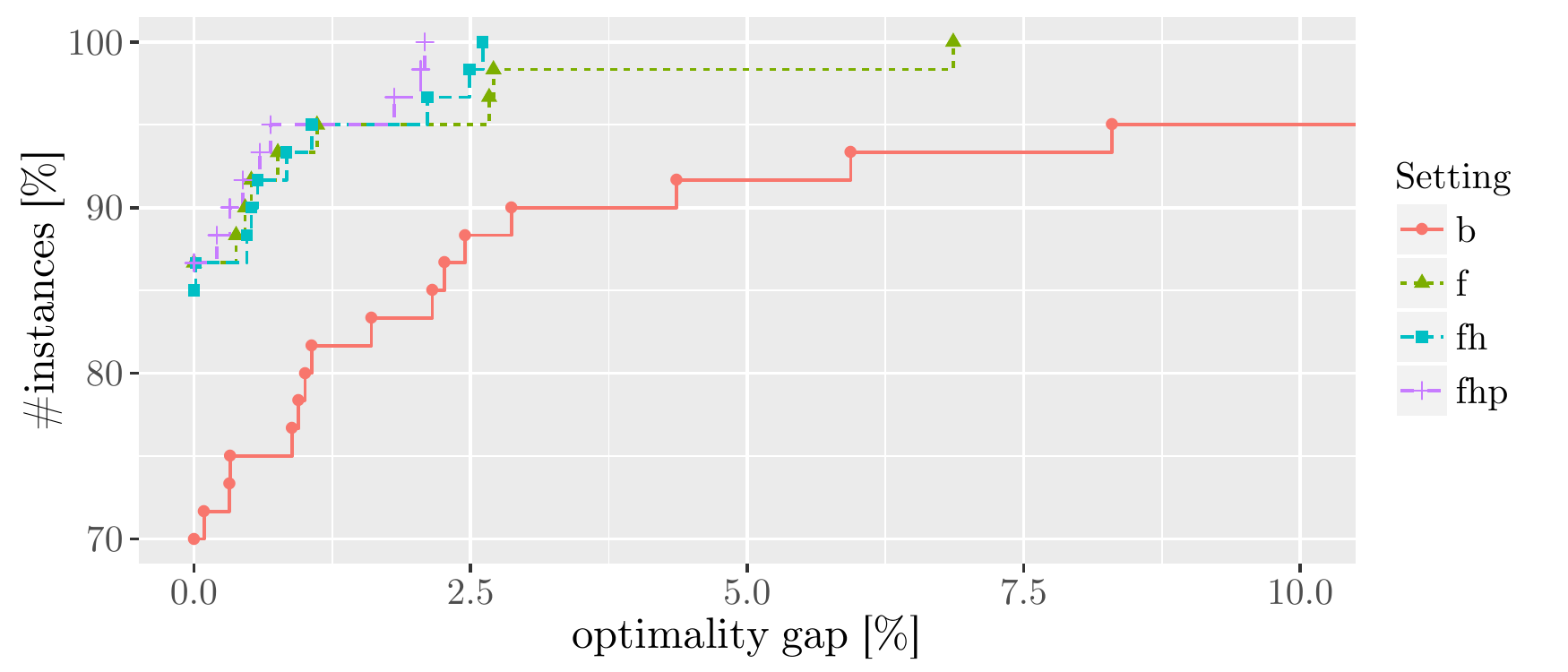}
\caption{Optimality gap performance of \fourth}\label{fig:gap4}
\end{subfigure}
\caption{Performance profiles of the attained optimality gaps for~\first-\fourth\ and different settings}\label{fig:gap}
\end{figure}

From the
results reported in the performance profiles two relevant conclusions can de drawn:
First, in terms of running times and attained gaps,
the \fourth\ formulation seems to be the most effective one. With the \fourth\ model
runtimes below ten seconds can be achieved for more than 70\% of the instances, while using \first, only around 10\%
of the instances can be solved to optimality within ten seconds.
Furthermore, when looking at the attained gaps, the situation is similar;
while the branch-and-cut based on the \fourth\ model is capable of
attaining optimality gaps below 2.5\% for all instances
with the \all\ configuration,
the branch-and-cut corresponding to \first\ computes
such gaps only for about 25\% of the instances with the \all\ configuration.

Secondly,
the different algorithmic enhancements do improve the
overall performance of the corresponding branch-and-cut algorithm.
In other words, by going from \basic\ to \all\ we get shorter runtimes
and smaller optimality gaps. 
Moreover, the impact of such enhancements, seems to be stronger
for the \third\ and \fourth\ formulations than for the \first\ and \second\.
Additionally, when focusing on the plots corresponding to the \third\ and \fourth\ models, it seems that incorporating the separation on fractional solutions
brings the best marginal improvement (i.e., \basic\ with respect to \frac),
specially in terms of the attained gaps.
However, in both cases, the results show that the approach
combining all features, \all, presents the best performance
(in terms of runtimes and attained gaps).
Due to the dominance of the \fourth\ formulation with the \all\ configuration, 
the results presented in remainder of this section correspond to this setting.

\subsection{Further details on algorithmic performance and solutions nature}

In Tables~\ref{ta:ta11}-\ref{ta:ta23}, we report detailed results for each instance set using formulation \fourth\ with configuration \all;
each instance set is induced by a particular setting of $r$, $R$ and $\theta$
(as explained in Section~\ref{subsec:instdesc}).
In these Tables, we report for each instance 
(identified by the number in column ``id'') the number of nodes 
(column ``$|V|$'', recall that in these instances we have $\fac=\cust=V$), 
the number of $f_{\cdot\cdot}=1$ (column ``$\#C1$''), 
the number of fractional $f_{\cdot\cdot} >0$ (column ``$\#CP$''), 
the runtime in seconds (column ``$t [s]$'', with TL indicating that the timelimit of 600 seconds was reached), 
the upper bound (column ``$UB$''), 
the value of the best solution found (column ``$z^*$'', in bold when it corresponds to the optimum value), 
the optimality gap (column ``$g [\%]$''), 
the number of nodes in the branch-and-bound tree (column ``$\#BBn$''), 
the time spent at the root node (column ``$t_r [s]$''), 
the upper bound at the root node (column ``$UB_r$''), 
the optimality gap at the root node 
(column ``$g_r [\%]$'', calculated as $100 \cdot (UB_r-z^*)/(z^*)$), 
the runtime of the starting heuristic (column ``$t_H [s]$'')
the value of the starting heuristic solution (column ``$z_H$'', in bold when it coincides to the optimum value), 
the primal gap between this solution and the best solution found 
(column ``$g_H [\%]$'', calculated as $(z_H-z^*)/(z_H)$), 
the number of locations with more than one opened facility in the best found solution (column ``$\#CL$''), 
and the maximum number of facilities opened at a single location in the best found solution (column ``$mCL$'').

From the results reported in this table, we can observe that the computational
difficulty is strongly influenced by the value of $|V|$ (i.e., size of the problem), instances with $|V| \geq 500$ can rarely be solved to optimality
(as can be seen from column ``$t [s]$'' and ``$g [\%]$''). 
Moreover, for the case of the instance set \seta,
one of the instances with $|V| = 900$ could not be solved due to memory limit issues.
Additionally, it is interesting to point out that, for a given value of $|V|$,
increasing the value of $K$ has a clear effect on the effectiveness of the algorithm.
On the one hand, for the smaller instances ($|V|\leq 400$), increasing $K$ results
in an increase of the runtimes (which is particularly clear for $|V| = 400$);
on the other hand, for larger instances (which typically reach the timelimit),
increasing the value of $K$ results in an improvement of the attained optimality gaps
(as can be seen from column ``$g [\%]$'').
Along the same line, the number of explored branch-and-bound nodes
(shown in column ``$\#BBn$''), also presents an interesting behavior. 
For small instances ($|V|\leq 200$), very few nodes are explored,
most likely because the initialization and the separation allow to compute
very tight dual bounds at a very early stage of the optimization process.
Likewise, for larger instances ($|V|\geq 700$), also very few nodes are explored;
however, in these cases, it is due to the large size of the induced linear programming models which results in a more time consuming separation process.
Such behavior is verified by the by longer runtimes required to process the root node
(column ``$t_r [s]$'') and the poorer quality of the root-node solutions
(which can be seen from columns ``$UB_r$'' and ``$g_r [\%]$'').
Therefore, it is only for intermediate size instances 
($300\leq |V|\leq 600$), that more branch-and-bound nodes are explored;
this is due to the moderate size of the resulting linear programming models 
and a less (computationally) expensive cut separation.

From columns ``$t_H [s]$'', ``$z^H$''  and ``$g_H [\%]$'',
we can clearly observe that the implemented starting heuristic
is capable of computing remarkably good solutions 
for the six groups of instances.
Moreover, on the contrary to the measures discussed in the previous paragraph,
the performance of the starting heuristic, 
specially the quality of the computed solutions (measured by the values reported in column ``$g_H [\%]$'')  seems to be insensitive to the values of $|V|$ and $K$ .
Recall that the starting heuristic builds upon one of the 
approaches outlined in~\cite{berman2018multiple},
which is based on the greedy  approximation algorithm of~\cite[][]{nemhauser1978analysis} for submodular optimization problems.

In~\cite{berman2018multiple}, the authors test their approach only on
instances from set \seta.
When comparing their results with those attained by our approach,
and reported in Table~\ref{ta:ta11}, we observe the following facts:
(i) while in~\cite{berman2018multiple}, optimality is proven
for 5 instances, we manage to prove it for 18 instances;
(ii) we improve the solution value obtained by~\cite{berman2018multiple}
for 7 additional instances; 
and, (iii) our primal-dual nature of our approach allows to always
account with a certificate of quality of the attained solution.

\paragraph{Analyzing the solution-characteristics}
We will now analyze and discuss key characteristics of the computed solutions,
and how they are influenced by the instances structure, which is given by $r$, $R$ and $\theta$, and $|V|$ and $K$.

Tables~\ref{ta:ta11},~\ref{ta:ta12} and~\ref{ta:ta13} report the solutions of instance sets \seta, \setb\ and \setc, respectively; 
hence, in thee instances we have that $R=20$, $r=5$,
and their only difference is the value of $\theta$ (0.2, 0.5 and 0.8 respectively).
From columns $z^*$, 
we can clearly see that the computed solutions in these three sets have very similar objective values.
The deployment of the co-locations deploy on the attained
solutions is also very similar; the values reported in columns  $\#CL$ and $mCL$,
let us conclude that different values of $\theta$ do not necessarily
influence on the need of co-locating facilities at different number locations (columns $\#CL$),
nor on the number of co-located facilities (columns  $mCL$).
In a sense, this result is counterintuitive;
due to the definition of~\ref{eq:cov},
one would expect that by increasing the value of $\theta$,
we would reduce the number of location where more than one facility is opened
(and, complementary, the number of opened facilities in such locations).

On the contrary to the above described behavior,
for the instance sets \setd, \sete\ and \setf\ (which are given by $R=25$ and $r=10$),
it is possible observe a moderate influence of the value of $\theta$,
as can be seen from Tables~\ref{ta:ta21},~\ref{ta:ta22} and~\ref{ta:ta23},
respectively.
As expected, smaller values of $\theta$ ($\theta = 0.2$) lead 
to solutions with more locations hosting multiple facilities
and more facilities located at those locations (see columns $\#CL$ and $mCL$),
when compared to having greater values of $\theta$ ($\theta = 0.8$).
Furthermore, the influence of $\theta$ 
can be also observed when comparing the attained
objective function values (columns ``$z^*$'');
as expected, greater values of $\theta$ lead to less expensive solutions,
as co-location and joint coverage are less emphasized 
(as when having  smaller values of $\theta$).

The difference in the behavior of the solutions
of instance sets \seta, \setb\ and \setc,
with respect to those of \setd, \sete\ and \setf,
is likely to be explained by the different values of $r$ and $R$.
While for the first group of instances $R$ is four times larger
than $r$, for the second, $R$ is only 2.5 times larger than $r$.
Additionally, for the first group, the value of $r$ is half
the value of $r$ for second group. Thus, for the first group of instances, there is a smaller number of facility-customer-pairs which would result in full coverage (see $\#C1$). For example, in the instance number 40, there are 3552 such candidates when $r=5$, while there are $14636$ for $r=10$. Moreover, also the combinations providing partial coverage are most of the time more numerous in the second group of instances. Thus, the first group of instances gives less choices and the problem resembles a little more a
classical maximum coverage problem, while the second group with more available connections potentially allows more exploitation of the benefits of partial coverage which are enhanced by the possibility of having more than one facility located at a given location.


\begin{landscape}
\begin{table}[ht]
\centering
\caption{Detailed results for instance set \seta.\label{ta:ta11}} 
\begingroup\scriptsize
\begin{tabular}{lllll|rrrrrrrrrrrrr}
  \toprule
  id & $|V|$ & K & $\#C1$ & $\#CP$ & $t [s]$ & $UB$ & $z^*$ & $g [\%]$ & $\#BBn$ & $t_r [s]$ & $UB_r$ & $g_r [\%]$ & $t_H [s]$ & $z^H$ & $g_H [\%]$ & $\#CL$ & $mCL$ \\ \midrule
1 & 100 & 5 & 114 & 64 & 0.01 &  14.60000 & \textbf{ 14.60000} & 0.000 & 0 & 0.01 & 14.60000 & 0.000 & 0.00 & \textbf{ 14.60000} & 0.000 & 0 & 1 \\ 
  2 & 100 & 10 & 124 & 76 & 0.02 &  26.79200 & \textbf{ 26.79200} & 0.000 & 0 & 0.02 & 26.79200 & 0.000 & 0.00 & \textbf{ 26.79200} & 0.000 & 0 & 1 \\ 
  3 & 100 & 10 & 116 & 100 & 0.07 &  25.65333 & \textbf{ 25.65333} & 0.000 & 0 & 0.07 & 25.65333 & 0.000 & 0.00 & \textbf{ 25.65333} & 0.000 & 0 & 1 \\ 
  4 & 100 & 20 & 106 & 56 & 0.05 &  35.43200 & \textbf{ 35.43200} & 0.000 & 0 & 0.05 & 35.43200 & 0.000 & 0.00 & \textbf{ 35.43200} & 0.000 & 0 & 1 \\ 
  5 & 100 & 33 & 130 & 114 & 0.14 &  62.21778 & \textbf{ 62.21778} & 0.000 & 1 & 0.14 & 62.21778 & 0.000 & 0.00 & \textbf{ 62.21778} & 0.000 & 0 & 1 \\ 
  6 & 200 & 5 & 296 & 430 & 0.02 &  30.13333 & \textbf{ 30.13333} & 0.000 & 0 & 0.02 & 30.13333 & 0.000 & 0.00 & \textbf{ 30.13333} & 0.000 & 0 & 1 \\ 
  7 & 200 & 10 & 292 & 592 & 0.26 &  50.48124 & \textbf{ 50.48124} & 0.000 & 4 & 0.25 & 50.56167 & 0.159 & 0.01 & \textbf{ 50.48124} & 0.000 & 0 & 1 \\ 
  8 & 200 & 20 & 278 & 418 & 0.13 &  69.79514 & \textbf{ 69.79514} & 0.000 & 0 & 0.13 & 69.79514 & 0.000 & 0.00 & \textbf{ 69.79514} & 0.000 & 0 & 1 \\ 
  9 & 200 & 40 & 302 & 502 & 1.37 & 118.10412 & \textbf{118.10412} & 0.000 & 36 & 1.00 & 118.35876 & 0.216 & 0.02 & 117.95745 & 0.124 & 3 & 2 \\ 
  10 & 200 & 67 & 328 & 900 & 58.07 & 158.93399 & \textbf{158.93399} & 0.000 & 1246 & 6.02 & 160.91652 & 1.247 & 0.09 & \textbf{158.93399} & 0.000 & 4 & 3 \\ 
  11 & 300 & 5 & 534 & 2358 & 0.61 &  62.52753 & \textbf{ 62.52753} & 0.000 & 10 & 0.44 & 63.20786 & 1.088 & 0.01 & \textbf{ 62.52753} & 0.000 & 0 & 1 \\ 
  12 & 300 & 10 & 474 & 1654 & 0.37 &  80.00622 & \textbf{ 80.00622} & 0.000 & 5 & 0.33 & 80.11857 & 0.140 & 0.00 & \textbf{ 80.00622} & 0.000 & 0 & 1 \\ 
  13 & 300 & 30 & 500 & 2004 & 29.07 & 154.42877 & \textbf{154.42877} & 0.000 & 475 & 4.08 & 157.06122 & 1.705 & 0.06 & 154.42506 & 0.002 & 0 & 1 \\ 
  14 & 300 & 60 & 502 & 1586 & \textbf{TL} & 207.97964 & 207.54577 & 0.209 & 4575 & 12.74 & 209.92442 & 1.146 & 0.15 & 206.64859 & 0.434 & 0 & 1 \\ 
  15 & 300 & 100 & 518 & 2000 & \textbf{TL} & 252.19124 & 250.45467 & 0.693 & 1327 & 28.12 & 253.87769 & 1.367 & 0.33 & 250.32876 & 0.050 & 2 & 2 \\ 
  16 & 400 & 5 & 874 & 6918 & 2.53 & 108.96806 & \textbf{108.96806} & 0.000 & 26 & 1.16 & 111.27877 & 2.121 & 0.01 & 108.03455 & 0.864 & 0 & 1 \\ 
  17 & 400 & 10 & 840 & 5902 & 4.37 & 147.92027 & \textbf{147.92027} & 0.000 & 55 & 1.89 & 150.69428 & 1.875 & 0.02 & \textbf{147.92027} & 0.000 & 0 & 1 \\ 
  18 & 400 & 40 & 738 & 4070 & \textbf{TL} & 243.01593 & 238.04555 & 2.088 & 1307 & 24.92 & 246.37574 & 3.499 & 0.10 & 237.38353 & 0.279 & 1 & 2 \\ 
  19 & 400 & 80 & 816 & 4932 & \textbf{TL} & 329.00121 & 323.15514 & 1.809 & 170 & 101.76 & 329.23856 & 1.883 & 0.94 & 323.07301 & 0.025 & 2 & 2 \\ 
  20 & 400 & 133 & 760 & 4170 & \textbf{TL} & 375.11226 & 367.57563 & 2.050 & 156 & 147.06 & 375.25593 & 2.089 & 1.50 & 366.89330 & 0.186 & 3 & 2 \\ 
  21 & 500 & 5 & 1186 & 12832 & 2.26 & 149.16160 & \textbf{149.16160} & 0.000 & 13 & 1.53 & 150.93946 & 1.192 & 0.01 & \textbf{149.16160} & 0.000 & 0 & 1 \\ 
  22 & 500 & 10 & 1090 & 9960 & 38.68 & 178.49851 & \textbf{178.49851} & 0.000 & 278 & 4.35 & 184.47554 & 3.349 & 0.03 & 176.32300 & 1.234 & 0 & 1 \\ 
  23 & 500 & 50 & 1196 & 11222 & \textbf{TL} & 377.38754 & 363.88111 & 3.712 & 157 & 83.84 & 377.80670 & 3.827 & 0.75 & 363.65744 & 0.062 & 3 & 2 \\ 
  24 & 500 & 100 & 1174 & 11328 & \textbf{TL} & 451.08139 & 439.76165 & 2.574 & 21 & 303.16 & 451.40047 & 2.647 & 1.79 & 439.13596 & 0.142 & 2 & 2 \\ 
  25 & 500 & 167 & 1258 & 13962 & \textbf{TL} & 489.27113 & 482.60720 & 1.381 & 8 & 469.44 & 489.36607 & 1.400 & 5.67 & 482.41662 & 0.040 & 4 & 2 \\ 
  26 & 600 & 5 & 1706 & 30160 & 76.97 & 221.74392 & \textbf{221.74392} & 0.000 & 167 & 6.78 & 234.29056 & 5.658 & 0.04 & \textbf{221.74392} & 0.000 & 0 & 1 \\ 
  27 & 600 & 10 & 1764 & 32768 & \textbf{TL} & 323.17551 & 305.54231 & 5.771 & 144 & 34.65 & 327.64050 & 7.232 & 0.09 & 305.49935 & 0.014 & 1 & 2 \\ 
  28 & 600 & 60 & 1716 & 33584 & \textbf{TL} & 519.85051 & 503.69890 & 3.207 & 3 & 499.19 & 520.03866 & 3.244 & 2.68 & 503.69889 & 0.000 & 1 & 2 \\ 
  29 & 600 & 120 & 1714 & 28968 & \textbf{TL} & 575.12023 & 560.94685 & 2.527 & 0 & 604.38 & 575.12023 & 2.527 & 4.96 & 560.88227 & 0.012 & 5 & 2 \\ 
  30 & 600 & 200 & 1570 & 21010 & \textbf{TL} & 593.24141 & 590.50481 & 0.463 & 0 & 604.15 & 593.24141 & 0.463 & 12.59 & 590.27327 & 0.039 & 6 & 2 \\ 
  31 & 700 & 5 & 2542 & 64420 & \textbf{TL} & 331.11567 & 314.32281 & 5.343 & 221 & 14.60 & 342.73355 & 9.039 & 0.04 & 314.32281 & 0.000 & 0 & 1 \\ 
  32 & 700 & 10 & 2218 & 50474 & \textbf{TL} & 397.36246 & 374.87069 & 6.000 & 76 & 32.41 & 398.47059 & 6.295 & 0.13 & 374.87069 & 0.000 & 0 & 1 \\ 
  33 & 700 & 70 & 2268 & 49250 & \textbf{TL} & 640.18003 & 621.39845 & 3.022 & 0 & 607.97 & 640.18003 & 3.022 & 3.40 & 620.85758 & 0.087 & 2 & 2 \\ 
  34 & 700 & 140 & 2450 & 58282 & \textbf{TL} & 686.32210 & 679.62662 & 0.985 & 0 & 627.53 & 686.32210 & 0.985 & 4.92 & 679.45459 & 0.025 & 1 & 2 \\ 
  35 & 800 & 5 & 3424 & 143422 & \textbf{TL} & 498.76943 & 460.69826 & 8.264 & 22 & 72.53 & 500.36055 & 8.609 & 0.05 & 460.69826 & 0.000 & 0 & 1 \\ 
  36 & 800 & 10 & 3052 & 91980 & \textbf{TL} & 517.48254 & 479.24997 & 7.978 & 25 & 68.11 & 517.94425 & 8.074 & 0.19 & 479.24997 & 0.000 & 0 & 1 \\ 
  37 & 800 & 80 & 2762 & 71842 & \textbf{TL} & 745.44345 & 727.12707 & 2.519 & 0 & 604.20 & 745.44345 & 2.519 & 4.68 & 726.84905 & 0.038 & 0 & 1 \\ 
  38 & 900 & 5 & 4520 & 224438 & \textbf{TL} & 619.22839 & 563.03405 & 9.981 & 9 & 108.97 & 620.35317 & 10.180 & 0.13 & 561.05866 & 0.352 & 0 & 1 \\ 
  39 & 900 & 10 & 4450 & 227826 & \textbf{TL} & 724.97571 & 682.84189 & 6.170 & 7 & 204.66 & 725.47869 & 6.244 & 0.20 & 682.84189 & 0.000 & 0 & 1 \\ 
  40 & 900 & 80 & * & * & \multicolumn{13}{c}{\emph{not solution available due to memory limit}} \\ 
   \bottomrule
\end{tabular}
\endgroup
\end{table}
\begin{table}[ht]
\centering
\caption{Detailed results for instance set \setb.\label{ta:ta12}} 
\begingroup\scriptsize
\begin{tabular}{lllll|rrrrrrrrrrrrr}
  \toprule
  id & $|V|$ & K & $\#C1$ & $\#CP$ & $t [s]$ & $UB$ & $z^*$ & $g [\%]$ & $\#BBn$ & $t_r [s]$ & $UB_r$ & $g_r [\%]$ & $t_H [s]$ & $z^H$ & $g_H [\%]$ & $\#CL$ & $mCL$ \\ \midrule
1 & 100 & 5 & 114 & 64 & 0.01 &  14.60000 & \textbf{ 14.60000} & 0.000 & 0 & 0.01 & 14.60000 & 0.000 & 0.00 & \textbf{ 14.60000} & 0.000 & 0 & 1 \\ 
  2 & 100 & 10 & 124 & 76 & 0.02 &  26.72000 & \textbf{ 26.72000} & 0.000 & 0 & 0.02 & 26.72000 & 0.000 & 0.00 & \textbf{ 26.72000} & 0.000 & 0 & 1 \\ 
  3 & 100 & 10 & 116 & 100 & 0.04 &  25.63333 & \textbf{ 25.63333} & 0.000 & 0 & 0.04 & 25.63333 & 0.000 & 0.00 & \textbf{ 25.63333} & 0.000 & 0 & 1 \\ 
  4 & 100 & 20 & 106 & 56 & 0.05 &  35.42000 & \textbf{ 35.42000} & 0.000 & 0 & 0.05 & 35.42000 & 0.000 & 0.00 & \textbf{ 35.42000} & 0.000 & 0 & 1 \\ 
  5 & 100 & 33 & 130 & 114 & 0.13 &  62.11111 & \textbf{ 62.11111} & 0.000 & 1 & 0.13 & 62.11111 & 0.000 & 0.00 & \textbf{ 62.11111} & 0.000 & 0 & 1 \\ 
  6 & 200 & 5 & 296 & 430 & 0.02 &  30.13333 & \textbf{ 30.13333} & 0.000 & 0 & 0.02 & 30.13333 & 0.000 & 0.00 & \textbf{ 30.13333} & 0.000 & 0 & 1 \\ 
  7 & 200 & 10 & 292 & 592 & 0.22 &  50.32578 & \textbf{ 50.32578} & 0.000 & 0 & 0.22 & 50.32578 & 0.000 & 0.00 & \textbf{ 50.32578} & 0.000 & 0 & 1 \\ 
  8 & 200 & 20 & 278 & 418 & 0.14 &  69.66030 & \textbf{ 69.66030} & 0.000 & 0 & 0.14 & 69.66030 & 0.000 & 0.01 & \textbf{ 69.66030} & 0.000 & 0 & 1 \\ 
  9 & 200 & 40 & 302 & 502 & 0.63 & 117.59007 & \textbf{117.59007} & 0.000 & 20 & 0.54 & 117.65543 & 0.056 & 0.04 & \textbf{117.59007} & 0.000 & 3 & 2 \\ 
  10 & 200 & 67 & 328 & 900 & 2.44 & 157.89400 & \textbf{157.89400} & 0.000 & 19 & 1.80 & 158.13614 & 0.153 & 0.07 & 157.77459 & 0.076 & 3 & 3 \\ 
  11 & 300 & 5 & 534 & 2358 & 0.35 &  62.07970 & \textbf{ 62.07970} & 0.000 & 1 & 0.35 & 62.07970 & 0.000 & 0.00 & \textbf{ 62.07970} & 0.000 & 0 & 1 \\ 
  12 & 300 & 10 & 474 & 1654 & 0.34 &  79.62889 & \textbf{ 79.62889} & 0.000 & 0 & 0.34 & 79.62889 & 0.000 & 0.00 & \textbf{ 79.62889} & 0.000 & 0 & 1 \\ 
  13 & 300 & 30 & 500 & 2004 & 3.57 & 152.75187 & \textbf{152.75187} & 0.000 & 45 & 2.15 & 153.45741 & 0.462 & 0.05 & 152.72967 & 0.015 & 1 & 2 \\ 
  14 & 300 & 60 & 502 & 1586 & 12.33 & 205.58821 & \textbf{205.58821} & 0.000 & 143 & 5.23 & 206.21113 & 0.303 & 0.13 & 204.76453 & 0.402 & 0 & 1 \\ 
  15 & 300 & 100 & 518 & 2000 & 99.1 & 247.41430 & \textbf{247.41430} & 0.000 & 1160 & 11.22 & 248.39762 & 0.397 & 0.19 & 246.78161 & 0.256 & 2 & 2 \\ 
  16 & 400 & 5 & 874 & 6918 & 1.24 & 107.40504 & \textbf{107.40504} & 0.000 & 8 & 0.98 & 108.24489 & 0.782 & 0.01 & 106.29763 & 1.042 & 0 & 1 \\ 
  17 & 400 & 10 & 840 & 5902 & 2.12 & 145.77517 & \textbf{145.77517} & 0.000 & 14 & 1.48 & 146.76176 & 0.677 & 0.01 & \textbf{145.77517} & 0.000 & 0 & 1 \\ 
  18 & 400 & 40 & 738 & 4070 & \textbf{TL} & 234.84788 & 233.81358 & 0.442 & 2675 & 14.13 & 237.47458 & 1.566 & 0.17 & 233.32576 & 0.209 & 0 & 1 \\ 
  19 & 400 & 80 & 816 & 4932 & \textbf{TL} & 318.65316 & 317.62369 & 0.324 & 848 & 49.02 & 319.45887 & 0.578 & 0.53 & 317.23561 & 0.122 & 6 & 2 \\ 
  20 & 400 & 133 & 760 & 4170 & \textbf{TL} & 364.64616 & 362.48410 & 0.596 & 250 & 61.27 & 365.12343 & 0.728 & 1.84 & 361.74052 & 0.206 & 4 & 2 \\ 
  21 & 500 & 5 & 1186 & 12832 & 1.17 & 146.17600 & \textbf{146.17600} & 0.000 & 0 & 1.17 & 146.17600 & 0.000 & 0.02 & \textbf{146.17600} & 0.000 & 0 & 1 \\ 
  22 & 500 & 10 & 1090 & 9960 & 14.48 & 174.86157 & \textbf{174.86157} & 0.000 & 92 & 3.64 & 177.70824 & 1.628 & 0.04 & 172.41694 & 1.418 & 0 & 1 \\ 
  23 & 500 & 50 & 1196 & 11222 & \textbf{TL} & 360.21888 & 354.87291 & 1.506 & 306 & 64.50 & 361.15181 & 1.769 & 0.45 & 354.03810 & 0.236 & 0 & 1 \\ 
  24 & 500 & 100 & 1174 & 11328 & \textbf{TL} & 434.69475 & 430.74256 & 0.918 & 142 & 100.63 & 435.07190 & 1.005 & 1.23 & 430.71258 & 0.007 & 3 & 2 \\ 
  25 & 500 & 167 & 1258 & 13962 & \textbf{TL} & 478.22622 & 476.03060 & 0.461 & 55 & 242.32 & 478.31624 & 0.480 & 5.67 & 476.03059 & 0.000 & 3 & 2 \\ 
  26 & 600 & 5 & 1706 & 30160 & 26.56 & 214.31495 & \textbf{214.31495} & 0.000 & 70 & 5.54 & 221.24374 & 3.233 & 0.06 & \textbf{214.31495} & 0.000 & 0 & 1 \\ 
  27 & 600 & 10 & 1764 & 32768 & \textbf{TL} & 299.53133 & 290.93944 & 2.953 & 181 & 20.42 & 304.76118 & 4.751 & 0.05 & 290.34134 & 0.206 & 0 & 1 \\ 
  28 & 600 & 60 & 1716 & 33584 & \textbf{TL} & 495.15376 & 487.18178 & 1.636 & 22 & 229.52 & 495.38495 & 1.684 & 1.70 & 486.87124 & 0.064 & 0 & 1 \\ 
  29 & 600 & 120 & 1714 & 28968 & \textbf{TL} & 553.70093 & 547.87855 & 1.063 & 19 & 297.53 & 553.83774 & 1.088 & 2.77 & 547.72427 & 0.028 & 3 & 3 \\ 
  30 & 600 & 200 & 1570 & 21010 & \textbf{TL} & 584.76703 & 583.39450 & 0.235 & 3 & 539.43 & 584.81751 & 0.244 & 8.68 & 583.26436 & 0.022 & 4 & 2 \\ 
  31 & 700 & 5 & 2542 & 64420 & 560.85 & 302.45176 & \textbf{302.45176} & 0.000 & 434 & 16.12 & 319.85636 & 5.755 & 0.06 & \textbf{302.45176} & 0.000 & 0 & 1 \\ 
  32 & 700 & 10 & 2218 & 50474 & \textbf{TL} & 367.70970 & 357.83233 & 2.760 & 202 & 27.44 & 371.81872 & 3.909 & 0.09 & 357.83233 & 0.000 & 1 & 2 \\ 
  33 & 700 & 70 & 2268 & 49250 & \textbf{TL} & 609.70806 & 600.68651 & 1.502 & 6 & 421.93 & 609.81106 & 1.519 & 1.75 & 600.51548 & 0.028 & 1 & 2 \\ 
  34 & 700 & 140 & 2450 & 58282 & \textbf{TL} & 668.80616 & 665.46411 & 0.502 & 0 & 612.04 & 668.80616 & 0.502 & 10.67 & 665.35981 & 0.016 & 6 & 2 \\ 
  35 & 800 & 5 & 3424 & 143422 & \textbf{TL} & 457.12331 & 432.06141 & 5.801 & 32 & 35.08 & 457.97255 & 5.997 & 0.05 & 431.86545 & 0.045 & 0 & 1 \\ 
  36 & 800 & 10 & 3052 & 91980 & \textbf{TL} & 478.35183 & 453.18283 & 5.554 & 48 & 43.35 & 478.79706 & 5.652 & 0.22 & 453.18283 & 0.000 & 0 & 1 \\ 
  37 & 800 & 80 & 2762 & 71842 & \textbf{TL} & 712.12631 & 704.15161 & 1.133 & 1 & 583.37 & 712.12631 & 1.133 & 6.85 & 704.08931 & 0.009 & 1 & 2 \\ 
  38 & 900 & 5 & 4520 & 224438 & \textbf{TL} & 562.74327 & 526.27128 & 6.930 & 18 & 136.65 & 563.82824 & 7.136 & 0.13 & 526.27128 & 0.000 & 0 & 1 \\ 
  39 & 900 & 10 & 4450 & 227826 & \textbf{TL} & 663.37077 & 637.70118 & 4.025 & 8 & 185.74 & 664.02808 & 4.128 & 0.19 & 637.70118 & 0.000 & 0 & 1 \\ 
  40 & 900 & 90 & 3552 & 140956 & \textbf{TL} & 833.23948 & 825.40900 & 0.949 & 0 & 605.08 & 833.23948 & 0.949 & 4.00 & 824.80465 & 0.073 & 2 & 2 \\ 
   \bottomrule
\end{tabular}
\endgroup
\end{table}
\begin{table}[ht]
\centering
\caption{Detailed results for instance set \setc.\label{ta:ta13}} 
\begingroup\scriptsize
\begin{tabular}{lllll|rrrrrrrrrrrrr}
  \toprule
  id & $|V|$ & K & $\#C1$ & $\#CP$ & $t [s]$ & $UB$ & $z^*$ & $g [\%]$ & $\#BBn$ & $t_r [s]$ & $UB_r$ & $g_r [\%]$ & $t_H [s]$ & $z^H$ & $g_H [\%]$ & $\#CL$ & $mCL$ \\ \midrule
1 & 100 & 5 & 114 & 64 & 0 &  14.60000 & \textbf{ 14.60000} & 0.000 & 0 & 0.00 & 14.60000 & 0.000 & 0.00 & \textbf{ 14.60000} & 0.000 & 0 & 1 \\ 
  2 & 100 & 10 & 124 & 76 & 0.02 &  26.64800 & \textbf{ 26.64800} & 0.000 & 0 & 0.02 & 26.64800 & 0.000 & 0.00 & \textbf{ 26.64800} & 0.000 & 0 & 1 \\ 
  3 & 100 & 10 & 116 & 100 & 0.04 &  25.61333 & \textbf{ 25.61333} & 0.000 & 0 & 0.04 & 25.61333 & 0.000 & 0.00 & \textbf{ 25.61333} & 0.000 & 0 & 1 \\ 
  4 & 100 & 20 & 106 & 56 & 0.04 &  35.40800 & \textbf{ 35.40800} & 0.000 & 0 & 0.04 & 35.40800 & 0.000 & 0.00 & \textbf{ 35.40800} & 0.000 & 0 & 1 \\ 
  5 & 100 & 33 & 130 & 114 & 0.11 &  62.00444 & \textbf{ 62.00444} & 0.000 & 1 & 0.11 & 62.00444 & 0.000 & 0.00 & \textbf{ 62.00444} & 0.000 & 0 & 1 \\ 
  6 & 200 & 5 & 296 & 430 & 0.03 &  30.13333 & \textbf{ 30.13333} & 0.000 & 0 & 0.03 & 30.13333 & 0.000 & 0.01 & \textbf{ 30.13333} & 0.000 & 0 & 1 \\ 
  7 & 200 & 10 & 292 & 592 & 0.12 &  50.20587 & \textbf{ 50.20587} & 0.000 & 0 & 0.12 & 50.20587 & 0.000 & 0.00 & \textbf{ 50.20587} & 0.000 & 0 & 1 \\ 
  8 & 200 & 20 & 278 & 418 & 0.14 &  69.54412 & \textbf{ 69.54412} & 0.000 & 0 & 0.14 & 69.54412 & 0.000 & 0.01 & \textbf{ 69.54412} & 0.000 & 0 & 1 \\ 
  9 & 200 & 40 & 302 & 502 & 0.65 & 117.07603 & \textbf{117.07603} & 0.000 & 10 & 0.47 & 117.08809 & 0.010 & 0.03 & \textbf{117.07603} & 0.000 & 3 & 2 \\ 
  10 & 200 & 67 & 328 & 900 & 1.18 & 157.13121 & \textbf{157.13121} & 0.000 & 1 & 1.18 & 157.13121 & 0.000 & 0.07 & 157.11760 & 0.009 & 4 & 2 \\ 
  11 & 300 & 5 & 534 & 2358 & 0.15 &  61.63188 & \textbf{ 61.63188} & 0.000 & 0 & 0.15 & 61.63188 & 0.000 & 0.01 & \textbf{ 61.63188} & 0.000 & 0 & 1 \\ 
  12 & 300 & 10 & 474 & 1654 & 0.23 &  79.25156 & \textbf{ 79.25156} & 0.000 & 0 & 0.23 & 79.25156 & 0.000 & 0.01 & \textbf{ 79.25156} & 0.000 & 0 & 1 \\ 
  13 & 300 & 30 & 500 & 2004 & 1.56 & 151.10075 & \textbf{151.10075} & 0.000 & 5 & 1.49 & 151.16044 & 0.040 & 0.06 & 151.08654 & 0.009 & 1 & 2 \\ 
  14 & 300 & 60 & 502 & 1586 & 3.59 & 203.81306 & \textbf{203.81306} & 0.000 & 20 & 2.95 & 203.87750 & 0.032 & 0.21 & 203.15461 & 0.324 & 1 & 2 \\ 
  15 & 300 & 100 & 518 & 2000 & 13.02 & 244.77680 & \textbf{244.77680} & 0.000 & 97 & 6.01 & 244.96591 & 0.077 & 0.18 & 244.12746 & 0.266 & 3 & 3 \\ 
  16 & 400 & 5 & 874 & 6918 & 0.55 & 105.84202 & \textbf{105.84202} & 0.000 & 0 & 0.55 & 105.84202 & 0.000 & 0.02 & \textbf{105.84202} & 0.000 & 0 & 1 \\ 
  17 & 400 & 10 & 840 & 5902 & 0.9 & 143.63007 & \textbf{143.63007} & 0.000 & 0 & 0.90 & 143.63007 & 0.000 & 0.02 & \textbf{143.63007} & 0.000 & 0 & 1 \\ 
  18 & 400 & 40 & 738 & 4070 & 20.18 & 230.27463 & \textbf{230.27463} & 0.000 & 114 & 10.18 & 230.82854 & 0.241 & 0.26 & 229.51506 & 0.331 & 1 & 2 \\ 
  19 & 400 & 80 & 816 & 4932 & 42.98 & 312.67871 & \textbf{312.67871} & 0.000 & 107 & 24.40 & 312.92945 & 0.080 & 1.17 & 312.42754 & 0.080 & 2 & 2 \\ 
  20 & 400 & 133 & 760 & 4170 & 98.38 & 357.99786 & \textbf{357.99786} & 0.000 & 252 & 38.16 & 358.28921 & 0.081 & 2.15 & 357.35165 & 0.181 & 1 & 2 \\ 
  21 & 500 & 5 & 1186 & 12832 & 0.79 & 143.19040 & \textbf{143.19040} & 0.000 & 0 & 0.79 & 143.19040 & 0.000 & 0.03 & \textbf{143.19040} & 0.000 & 0 & 1 \\ 
  22 & 500 & 10 & 1090 & 9960 & 5.15 & 171.22463 & \textbf{171.22463} & 0.000 & 18 & 3.78 & 171.74588 & 0.304 & 0.06 & 168.88678 & 1.384 & 0 & 1 \\ 
  23 & 500 & 50 & 1196 & 11222 & 119.93 & 346.94339 & \textbf{346.94339} & 0.000 & 186 & 36.89 & 347.94152 & 0.288 & 0.65 & 346.04908 & 0.258 & 0 & 1 \\ 
  24 & 500 & 100 & 1174 & 11328 & \textbf{TL} & 422.49787 & 422.36762 & 0.031 & 648 & 60.20 & 423.10762 & 0.175 & 1.66 & 421.74503 & 0.148 & 2 & 3 \\ 
  25 & 500 & 167 & 1258 & 13962 & \textbf{TL} & 470.51314 & 470.46275 & 0.011 & 400 & 166.26 & 470.69439 & 0.049 & 8.47 & 470.14281 & 0.068 & 2 & 2 \\ 
  26 & 600 & 5 & 1706 & 30160 & 8.35 & 206.88598 & \textbf{206.88598} & 0.000 & 15 & 4.69 & 209.12648 & 1.083 & 0.05 & \textbf{206.88598} & 0.000 & 0 & 1 \\ 
  27 & 600 & 10 & 1764 & 32768 & 195.75 & 278.78555 & \textbf{278.78555} & 0.000 & 305 & 15.16 & 284.42413 & 2.023 & 0.10 & 278.73863 & 0.017 & 0 & 1 \\ 
  28 & 600 & 60 & 1716 & 33584 & \textbf{TL} & 474.72976 & 472.49510 & 0.473 & 157 & 134.12 & 475.10421 & 0.552 & 1.56 & 472.25887 & 0.050 & 1 & 2 \\ 
  29 & 600 & 120 & 1714 & 28968 & \textbf{TL} & 538.06209 & 536.95262 & 0.207 & 83 & 231.42 & 538.29481 & 0.250 & 3.30 & 536.77142 & 0.034 & 1 & 2 \\ 
  30 & 600 & 200 & 1570 & 21010 & \textbf{TL} & 577.59212 & 577.39683 & 0.034 & 76 & 303.27 & 577.68397 & 0.050 & 7.36 & 576.88482 & 0.089 & 1 & 2 \\ 
  31 & 700 & 5 & 2542 & 64420 & 69.28 & 290.58070 & \textbf{290.58070} & 0.000 & 76 & 10.41 & 298.01495 & 2.558 & 0.04 & \textbf{290.58070} & 0.000 & 0 & 1 \\ 
  32 & 700 & 10 & 2218 & 50474 & 215.16 & 340.81293 & \textbf{340.81293} & 0.000 & 249 & 18.60 & 346.15909 & 1.569 & 0.09 & \textbf{340.81293} & 0.000 & 1 & 2 \\ 
  33 & 700 & 70 & 2268 & 49250 & \textbf{TL} & 584.36914 & 581.92421 & 0.420 & 60 & 208.39 & 584.68282 & 0.474 & 2.18 & 581.26311 & 0.114 & 3 & 2 \\ 
  34 & 700 & 140 & 2450 & 58282 & \textbf{TL} & 653.81786 & 653.34890 & 0.072 & 14 & 457.78 & 653.91650 & 0.087 & 8.14 & 652.87986 & 0.072 & 2 & 2 \\ 
  35 & 800 & 5 & 3424 & 143422 & 467.3 & 403.50618 & \textbf{403.50618} & 0.000 & 113 & 32.17 & 416.18687 & 3.143 & 0.06 & \textbf{403.50618} & 0.000 & 0 & 1 \\ 
  36 & 800 & 10 & 3052 & 91980 & \textbf{TL} & 437.67703 & 427.42371 & 2.399 & 101 & 34.33 & 440.39391 & 3.035 & 0.16 & 425.81790 & 0.377 & 0 & 1 \\ 
  37 & 800 & 80 & 2762 & 71842 & \textbf{TL} & 685.88115 & 683.34521 & 0.371 & 8 & 421.27 & 685.97471 & 0.385 & 3.00 & 683.24213 & 0.015 & 2 & 2 \\ 
  38 & 900 & 5 & 4520 & 224438 & \textbf{TL} & 508.25104 & 490.73630 & 3.569 & 25 & 62.31 & 509.28211 & 3.779 & 0.25 & 490.73630 & 0.000 & 0 & 1 \\ 
  39 & 900 & 10 & 4450 & 227826 & \textbf{TL} & 604.44502 & 592.56047 & 2.006 & 25 & 98.40 & 605.03571 & 2.105 & 0.21 & 592.56047 & 0.000 & 0 & 1 \\ 
  40 & 900 & 90 & 3552 & 140956 & \textbf{TL} & 802.88504 & 800.98919 & 0.237 & 1 & 608.49 & 802.88504 & 0.237 & 4.90 & 800.76111 & 0.028 & 1 & 2 \\ 
   \bottomrule
\end{tabular}
\endgroup
\end{table}
\begin{table}[ht]
\centering
\caption{Detailed results for instance set \setd.\label{ta:ta21}} 
\begingroup\scriptsize
\begin{tabular}{lllll|rrrrrrrrrrrrr}
  \toprule
  id & $|V|$ & K & $\#C1$ & $\#CP$ & $t [s]$ & $UB$ & $z^*$ & $g [\%]$ & $\#BBn$ & $t_r [s]$ & $UB_r$ & $g_r [\%]$ & $t_H [s]$ & $z^H$ & $g_H [\%]$ & $\#CL$ & $mCL$ \\ \midrule
1 & 100 & 5 & 138 & 60 & 0.01 &  17.53333 & \textbf{ 17.53333} & 0.000 & 0 & 0.01 & 17.53333 & 0.000 & 0.00 & \textbf{ 17.53333} & 0.000 & 0 & 1 \\ 
  2 & 100 & 10 & 152 & 78 & 0.04 &  31.79597 & \textbf{ 31.79597} & 0.000 & 0 & 0.04 & 31.79597 & 0.000 & 0.00 &  31.63828 & 0.498 & 1 & 2 \\ 
  3 & 100 & 10 & 136 & 110 & 0.04 &  32.00000 & \textbf{ 32.00000} & 0.000 & 0 & 0.04 & 32.00000 & 0.000 & 0.00 & \textbf{ 32.00000} & 0.000 & 0 & 1 \\ 
  4 & 100 & 20 & 126 & 66 & 0.04 &  43.55556 & \textbf{ 43.55556} & 0.000 & 0 & 0.04 & 43.55556 & 0.000 & 0.00 & \textbf{ 43.55556} & 0.000 & 0 & 1 \\ 
  5 & 100 & 33 & 174 & 116 & 0.18 &  70.43111 & \textbf{ 70.43111} & 0.000 & 0 & 0.18 & 70.43111 & 0.000 & 0.01 & \textbf{ 70.43111} & 0.000 & 1 & 2 \\ 
  6 & 200 & 5 & 408 & 628 & 0.1 &  41.22133 & \textbf{ 41.22133} & 0.000 & 0 & 0.10 & 41.22133 & 0.000 & 0.00 & \textbf{ 41.22133} & 0.000 & 0 & 1 \\ 
  7 & 200 & 10 & 414 & 906 & 0.35 &  67.98513 & \textbf{ 67.98513} & 0.000 & 2 & 0.33 & 68.01286 & 0.041 & 0.00 & \textbf{ 67.98513} & 0.000 & 0 & 1 \\ 
  8 & 200 & 20 & 360 & 658 & 0.26 &  93.38347 & \textbf{ 93.38347} & 0.000 & 0 & 0.26 & 93.38347 & 0.000 & 0.01 & \textbf{ 93.38347} & 0.000 & 0 & 1 \\ 
  9 & 200 & 40 & 418 & 686 & 4.27 & 140.75464 & \textbf{140.75464} & 0.000 & 118 & 1.89 & 141.30842 & 0.393 & 0.06 & 140.73299 & 0.015 & 1 & 2 \\ 
  10 & 200 & 67 & 526 & 1372 & 98.39 & 184.06172 & \textbf{184.06172} & 0.000 & 1378 & 8.88 & 185.75057 & 0.918 & 0.08 & 183.91174 & 0.082 & 3 & 2 \\ 
  11 & 300 & 5 & 1004 & 3968 & 0.62 &  96.19941 & \textbf{ 96.19941} & 0.000 & 5 & 0.59 & 96.47004 & 0.281 & 0.01 & \textbf{ 96.19941} & 0.000 & 0 & 1 \\ 
  12 & 300 & 10 & 814 & 2756 & 1.03 & 121.76729 & \textbf{121.76729} & 0.000 & 15 & 0.82 & 122.23208 & 0.382 & 0.01 & \textbf{121.76729} & 0.000 & 0 & 1 \\ 
  13 & 300 & 30 & 902 & 3342 & 48.08 & 200.24834 & \textbf{200.24834} & 0.000 & 474 & 9.66 & 202.68305 & 1.216 & 0.07 & 200.24169 & 0.003 & 0 & 1 \\ 
  14 & 300 & 60 & 804 & 2722 & \textbf{TL} & 247.77258 & 247.41956 & 0.143 & 2809 & 27.04 & 249.51119 & 0.845 & 0.27 & 247.19514 & 0.091 & 0 & 1 \\ 
  15 & 300 & 100 & 910 & 3470 & \textbf{TL} & 289.10723 & 285.98499 & 1.092 & 573 & 46.68 & 289.89205 & 1.366 & 0.45 & 285.62724 & 0.125 & 4 & 2 \\ 
  16 & 400 & 5 & 1974 & 13784 & 6.31 & 181.42063 & \textbf{181.42063} & 0.000 & 32 & 2.03 & 186.76530 & 2.946 & 0.02 & \textbf{181.42063} & 0.000 & 0 & 1 \\ 
  17 & 400 & 10 & 1738 & 11266 & 197.47 & 220.59393 & \textbf{220.59393} & 0.000 & 1298 & 6.95 & 227.75208 & 3.245 & 0.06 & 219.60947 & 0.448 & 0 & 1 \\ 
  18 & 400 & 40 & 1430 & 7758 & \textbf{TL} & 314.32964 & 308.97821 & 1.732 & 564 & 47.49 & 316.39468 & 2.400 & 0.24 & 308.83571 & 0.046 & 2 & 2 \\ 
  19 & 400 & 80 & 1644 & 9618 & \textbf{TL} & 379.62734 & 375.64126 & 1.061 & 182 & 106.09 & 380.00260 & 1.161 & 0.72 & 375.55034 & 0.024 & 3 & 2 \\ 
  20 & 400 & 133 & 1462 & 8014 & \textbf{TL} & 397.89812 & 397.15671 & 0.187 & 150 & 147.13 & 398.04788 & 0.224 & 2.69 & 397.08061 & 0.019 & 5 & 2 \\ 
  21 & 500 & 5 & 2918 & 27626 & 13.56 & 252.42390 & \textbf{252.42390} & 0.000 & 38 & 4.25 & 258.07630 & 2.239 & 0.03 & \textbf{252.42390} & 0.000 & 0 & 1 \\ 
  22 & 500 & 10 & 2496 & 21494 & \textbf{TL} & 289.36879 & 284.23641 & 1.806 & 606 & 10.15 & 296.97951 & 4.483 & 0.13 & 284.23641 & 0.000 & 0 & 1 \\ 
  23 & 500 & 50 & 2772 & 24348 & \textbf{TL} & 453.67706 & 443.81430 & 2.222 & 53 & 154.77 & 453.92340 & 2.278 & 0.54 & 443.60046 & 0.048 & 1 & 2 \\ 
  24 & 500 & 100 & 2810 & 23810 & \textbf{TL} & 494.49860 & 490.89659 & 0.734 & 33 & 208.13 & 494.53532 & 0.741 & 1.65 & 490.86898 & 0.006 & 3 & 3 \\ 
  25 & 500 & 167 & 3078 & 29866 & \textbf{TL} & 499.72439 & 499.69513 & 0.006 & 100 & 327.18 & 499.73327 & 0.008 & 7.88 & 499.64449 & 0.010 & 4 & 3 \\ 
  26 & 600 & 5 & 5086 & 70238 & \textbf{TL} & 389.77965 & 370.60402 & 5.174 & 108 & 22.70 & 393.75697 & 6.247 & 0.03 & 370.60402 & 0.000 & 0 & 1 \\ 
  27 & 600 & 10 & 5420 & 75078 & \textbf{TL} & 466.45326 & 446.98031 & 4.357 & 42 & 70.18 & 467.18071 & 4.519 & 0.10 & 445.93657 & 0.234 & 0 & 1 \\ 
  28 & 600 & 60 & 5338 & 78458 & \textbf{TL} & 583.44767 & 574.68799 & 1.524 & 3 & 459.65 & 583.61116 & 1.553 & 1.67 & 574.18047 & 0.088 & 3 & 2 \\ 
  29 & 600 & 120 & 4982 & 67654 & \textbf{TL} & 597.90520 & 597.17802 & 0.122 & 8 & 416.18 & 597.92809 & 0.126 & 4.69 & 596.95331 & 0.038 & 2 & 2 \\ 
  30 & 600 & 200 & 4284 & 48514 & 28.02 & 600.00000 & \textbf{600.00000} & 0.000 & 0 & 28.02 & 600.00000 & 0.000 & 3.74 & \textbf{600.00000} & 0.000 & 7 & 7 \\ 
  31 & 700 & 5 & 8624 & 152102 & \textbf{TL} & 543.71703 & 510.57203 & 6.492 & 37 & 68.55 & 545.00485 & 6.744 & 0.06 & 508.89662 & 0.329 & 0 & 1 \\ 
  32 & 700 & 10 & 7294 & 123992 & \textbf{TL} & 577.08107 & 553.66841 & 4.229 & 24 & 107.81 & 577.50451 & 4.305 & 0.24 & 552.92943 & 0.134 & 1 & 2 \\ 
  33 & 700 & 70 & 7110 & 121930 & \textbf{TL} & 692.47546 & 687.95411 & 0.657 & 2 & 541.67 & 692.55299 & 0.668 & 3.29 & 687.75055 & 0.030 & 1 & 2 \\ 
  34 & 700 & 140 & 8044 & 141988 & \textbf{TL} & 700.00000 & 699.79615 & 0.029 & 0 & 603.37 & 700.00000 & 0.029 & 10.28 & 699.77804 & 0.003 & 2 & 2 \\ 
  35 & 800 & 5 & 14838 & 312066 & \textbf{TL} & 697.75670 & 668.81970 & 4.327 & 7 & 133.40 & 698.54679 & 4.445 & 0.08 & 668.81970 & 0.000 & 0 & 1 \\ 
  36 & 800 & 10 & 10904 & 224412 & \textbf{TL} & 703.35094 & 677.49941 & 3.816 & 11 & 204.02 & 704.00226 & 3.912 & 0.20 & 677.25554 & 0.036 & 0 & 1 \\ 
  37 & 800 & 80 & 9072 & 186296 & \textbf{TL} & 794.07320 & 790.62270 & 0.436 & 0 & 606.86 & 794.07320 & 0.436 & 8.17 & 790.62265 & 0.000 & 0 & 1 \\ 
  38 & 900 & 5 & 21526 & 466312 & \textbf{TL} & 820.78375 & 788.21215 & 4.132 & 4 & 223.07 & 821.39778 & 4.210 & 0.20 & 788.21215 & 0.000 & 0 & 1 \\ 
  39 & 900 & 10 & 21248 & 471774 & \textbf{TL} & 856.26718 & 841.15019 & 1.797 & 3 & 283.25 & 856.53510 & 1.829 & 0.39 & 841.15012 & 0.000 & 0 & 1 \\ 
  40 & 900 & 90 & 14636 & 345298 & \textbf{TL} & 898.40506 & 897.00137 & 0.156 & 0 & 610.10 & 898.40506 & 0.156 & 5.69 & 896.88462 & 0.013 & 2 & 2 \\ 
   \bottomrule
\end{tabular}
\endgroup
\end{table}
\begin{table}[ht]
\centering
\caption{Detailed results for instance set \sete.\label{ta:ta22}} 
\begingroup\scriptsize
\begin{tabular}{lllll|rrrrrrrrrrrrr}
  \toprule
  id & $|V|$ & K & $\#C1$ & $\#CP$ & $t [s]$ & $UB$ & $z^*$ & $g [\%]$ & $\#BBn$ & $t_r [s]$ & $UB_r$ & $g_r [\%]$ & $t_H [s]$ & $z^H$ & $g_H [\%]$ & $\#CL$ & $mCL$ \\ \midrule
1 & 100 & 5 & 138 & 60 & 0.01 &  17.53333 & \textbf{ 17.53333} & 0.000 & 0 & 0.01 & 17.53333 & 0.000 & 0.00 & \textbf{ 17.53333} & 0.000 & 0 & 1 \\ 
  2 & 100 & 10 & 152 & 78 & 0.03 &  31.69748 & \textbf{ 31.69748} & 0.000 & 0 & 0.03 & 31.69748 & 0.000 & 0.00 &  31.57393 & 0.391 & 1 & 2 \\ 
  3 & 100 & 10 & 136 & 110 & 0.03 &  32.00000 & \textbf{ 32.00000} & 0.000 & 0 & 0.03 & 32.00000 & 0.000 & 0.00 & \textbf{ 32.00000} & 0.000 & 0 & 1 \\ 
  4 & 100 & 20 & 126 & 66 & 0.04 &  43.52222 & \textbf{ 43.52222} & 0.000 & 0 & 0.04 & 43.52222 & 0.000 & 0.00 & \textbf{ 43.52222} & 0.000 & 0 & 1 \\ 
  5 & 100 & 33 & 174 & 116 & 0.11 &  70.34444 & \textbf{ 70.34444} & 0.000 & 0 & 0.11 & 70.34444 & 0.000 & 0.01 & \textbf{ 70.34444} & 0.000 & 1 & 2 \\ 
  6 & 200 & 5 & 408 & 628 & 0.1 &  41.11333 & \textbf{ 41.11333} & 0.000 & 0 & 0.10 & 41.11333 & 0.000 & 0.00 & \textbf{ 41.11333} & 0.000 & 0 & 1 \\ 
  7 & 200 & 10 & 414 & 906 & 0.26 &  67.61570 & \textbf{ 67.61570} & 0.000 & 0 & 0.26 & 67.61570 & 0.000 & 0.01 &  67.44711 & 0.250 & 0 & 1 \\ 
  8 & 200 & 20 & 360 & 658 & 0.2 &  93.11467 & \textbf{ 93.11467} & 0.000 & 0 & 0.20 & 93.11467 & 0.000 & 0.01 & \textbf{ 93.11467} & 0.000 & 0 & 1 \\ 
  9 & 200 & 40 & 418 & 686 & 1.1 & 140.16165 & \textbf{140.16165} & 0.000 & 18 & 0.82 & 140.30235 & 0.100 & 0.03 & 139.76427 & 0.284 & 1 & 2 \\ 
  10 & 200 & 67 & 526 & 1372 & 7.34 & 182.75316 & \textbf{182.75316} & 0.000 & 110 & 3.40 & 183.20081 & 0.245 & 0.07 & 182.57485 & 0.098 & 5 & 2 \\ 
  11 & 300 & 5 & 1004 & 3968 & 0.36 &  95.44963 & \textbf{ 95.44963} & 0.000 & 3 & 0.35 & 95.53896 & 0.094 & 0.01 & \textbf{ 95.44963} & 0.000 & 0 & 1 \\ 
  12 & 300 & 10 & 814 & 2756 & 0.84 & 120.57956 & \textbf{120.57956} & 0.000 & 6 & 0.78 & 120.60948 & 0.025 & 0.02 & 120.45310 & 0.105 & 0 & 1 \\ 
  13 & 300 & 30 & 902 & 3342 & 6.45 & 198.40521 & \textbf{198.40521} & 0.000 & 32 & 4.41 & 199.15613 & 0.378 & 0.05 & 198.08131 & 0.164 & 0 & 1 \\ 
  14 & 300 & 60 & 804 & 2722 & 20.1 & 245.20222 & \textbf{245.20222} & 0.000 & 168 & 9.37 & 245.90362 & 0.286 & 0.26 & 244.24898 & 0.390 & 0 & 1 \\ 
  15 & 300 & 100 & 910 & 3470 & \textbf{TL} & 283.28413 & 283.16395 & 0.042 & 2353 & 17.30 & 284.29512 & 0.399 & 0.34 & 282.57714 & 0.208 & 4 & 2 \\ 
  16 & 400 & 5 & 1974 & 13784 & 4 & 177.56289 & \textbf{177.56289} & 0.000 & 19 & 2.09 & 180.59247 & 1.706 & 0.02 & \textbf{177.56289} & 0.000 & 0 & 1 \\ 
  17 & 400 & 10 & 1738 & 11266 & 26.74 & 216.97120 & \textbf{216.97120} & 0.000 & 192 & 6.02 & 220.71472 & 1.725 & 0.04 & \textbf{216.97120} & 0.000 & 0 & 1 \\ 
  18 & 400 & 40 & 1430 & 7758 & \textbf{TL} & 304.79825 & 304.49173 & 0.101 & 1755 & 32.50 & 307.60140 & 1.021 & 0.25 & 303.77542 & 0.236 & 3 & 2 \\ 
  19 & 400 & 80 & 1644 & 9618 & \textbf{TL} & 372.06546 & 371.85967 & 0.055 & 1366 & 45.23 & 372.97701 & 0.300 & 1.09 & 371.23688 & 0.168 & 3 & 2 \\ 
  20 & 400 & 133 & 1462 & 8014 & 246.9 & 395.68115 & \textbf{395.68115} & 0.000 & 341 & 91.60 & 395.88009 & 0.050 & 2.56 & 395.52344 & 0.040 & 4 & 2 \\ 
  21 & 500 & 5 & 2918 & 27626 & 5.4 & 246.63994 & \textbf{246.63994} & 0.000 & 14 & 3.17 & 248.88136 & 0.909 & 0.02 & \textbf{246.63994} & 0.000 & 0 & 1 \\ 
  22 & 500 & 10 & 2496 & 21494 & 186.94 & 277.37275 & \textbf{277.37275} & 0.000 & 500 & 9.22 & 285.04695 & 2.767 & 0.18 & \textbf{277.37275} & 0.000 & 0 & 1 \\ 
  23 & 500 & 50 & 2772 & 24348 & \textbf{TL} & 441.46235 & 436.65058 & 1.102 & 189 & 81.81 & 441.95910 & 1.216 & 0.60 & 436.48554 & 0.038 & 1 & 2 \\ 
  24 & 500 & 100 & 2810 & 23810 & \textbf{TL} & 488.30956 & 486.76514 & 0.317 & 114 & 127.53 & 488.41310 & 0.339 & 1.74 & 486.27600 & 0.101 & 4 & 2 \\ 
  25 & 500 & 167 & 3078 & 29866 & \textbf{TL} & 499.37030 & 499.36950 & 0.000 & 369 & 334.94 & 499.37952 & 0.002 & 4.58 & 499.24489 & 0.025 & 2 & 2 \\ 
  26 & 600 & 5 & 5086 & 70238 & \textbf{TL} & 361.65548 & 356.70251 & 1.389 & 181 & 23.46 & 371.64846 & 4.190 & 0.03 & 356.70251 & 0.000 & 0 & 1 \\ 
  27 & 600 & 10 & 5420 & 75078 & \textbf{TL} & 443.23151 & 431.38015 & 2.747 & 90 & 43.41 & 444.48826 & 3.039 & 0.10 & 431.38015 & 0.000 & 0 & 1 \\ 
  28 & 600 & 60 & 5338 & 78458 & \textbf{TL} & 569.99048 & 565.97899 & 0.709 & 27 & 281.00 & 570.11752 & 0.731 & 2.08 & 565.86250 & 0.021 & 1 & 2 \\ 
  29 & 600 & 120 & 4982 & 67654 & \textbf{TL} & 595.09777 & 594.54609 & 0.093 & 26 & 343.76 & 595.15516 & 0.102 & 3.74 & 594.08655 & 0.077 & 0 & 1 \\ 
  30 & 600 & 200 & 4284 & 48514 & 24.62 & 600.00000 & \textbf{600.00000} & 0.000 & 0 & 24.62 & 600.00000 & 0.000 & 1.96 & \textbf{600.00000} & 0.000 & 3 & 3 \\ 
  31 & 700 & 5 & 8624 & 152102 & \textbf{TL} & 512.46549 & 490.90752 & 4.391 & 60 & 34.01 & 513.49092 & 4.600 & 0.04 & 489.27721 & 0.333 & 0 & 1 \\ 
  32 & 700 & 10 & 7294 & 123992 & \textbf{TL} & 547.60250 & 533.11776 & 2.717 & 46 & 61.53 & 547.86623 & 2.766 & 0.27 & 533.11776 & 0.000 & 1 & 2 \\ 
  33 & 700 & 70 & 7110 & 121930 & \textbf{TL} & 682.62810 & 679.97642 & 0.390 & 9 & 417.71 & 682.75339 & 0.408 & 3.44 & 679.74519 & 0.034 & 1 & 2 \\ 
  34 & 700 & 140 & 8044 & 141988 & \textbf{TL} & 699.59917 & 699.59490 & 0.001 & 0 & 624.71 & 699.59917 & 0.001 & 7.37 & 699.45761 & 0.020 & 1 & 2 \\ 
  35 & 800 & 5 & 14838 & 312066 & \textbf{TL} & 660.25472 & 640.36231 & 3.106 & 11 & 85.58 & 660.87949 & 3.204 & 0.08 & 640.36231 & 0.000 & 0 & 1 \\ 
  36 & 800 & 10 & 10904 & 224412 & \textbf{TL} & 668.68489 & 649.23970 & 2.995 & 20 & 139.68 & 669.17449 & 3.070 & 0.25 & 649.22013 & 0.003 & 0 & 1 \\ 
  37 & 800 & 80 & 9072 & 186296 & \textbf{TL} & 785.75330 & 783.32037 & 0.311 & 0 & 614.14 & 785.75330 & 0.311 & 4.85 & 783.10236 & 0.028 & 3 & 2 \\ 
  38 & 900 & 5 & 21526 & 466312 & \textbf{TL} & 778.29020 & 755.52010 & 3.014 & 7 & 178.76 & 779.30087 & 3.148 & 0.10 & 755.52010 & 0.000 & 0 & 1 \\ 
  39 & 900 & 10 & 21248 & 471774 & \textbf{TL} & 825.87441 & 813.77446 & 1.487 & 5 & 277.99 & 826.24229 & 1.532 & 0.42 & 813.77440 & 0.000 & 0 & 1 \\ 
  40 & 900 & 90 & 14636 & 345298 & \textbf{TL} & 894.19899 & 893.13174 & 0.119 & 0 & 609.33 & 894.19899 & 0.119 & 5.87 & 892.75694 & 0.042 & 0 & 1 \\ 
   \bottomrule
\end{tabular}
\endgroup
\end{table}
\begin{table}[ht]
\centering
\caption{Detailed results for instance set \setf.\label{ta:ta23}} 
\begingroup\scriptsize
\begin{tabular}{lllll|rrrrrrrrrrrrr}
  \toprule
  id & $|V|$ & K & $\#C1$ & $\#CP$ & $t [s]$ & $UB$ & $z^*$ & $g [\%]$ & $\#BBn$ & $t_r [s]$ & $UB_r$ & $g_r [\%]$ & $t_H [s]$ & $z^H$ & $g_H [\%]$ & $\#CL$ & $mCL$ \\ \midrule
1 & 100 & 5 & 138 & 60 & 0.01 &  17.53333 & \textbf{ 17.53333} & 0.000 & 0 & 0.01 & 17.53333 & 0.000 & 0.00 & \textbf{ 17.53333} & 0.000 & 0 & 1 \\ 
  2 & 100 & 10 & 152 & 78 & 0.02 &  31.59899 & \textbf{ 31.59899} & 0.000 & 0 & 0.02 & 31.59899 & 0.000 & 0.00 &  31.50957 & 0.284 & 1 & 2 \\ 
  3 & 100 & 10 & 136 & 110 & 0.03 &  32.00000 & \textbf{ 32.00000} & 0.000 & 0 & 0.03 & 32.00000 & 0.000 & 0.00 & \textbf{ 32.00000} & 0.000 & 0 & 1 \\ 
  4 & 100 & 20 & 126 & 66 & 0.04 &  43.48889 & \textbf{ 43.48889} & 0.000 & 0 & 0.04 & 43.48889 & 0.000 & 0.00 & \textbf{ 43.48889} & 0.000 & 0 & 1 \\ 
  5 & 100 & 33 & 174 & 116 & 0.1 &  70.25778 & \textbf{ 70.25778} & 0.000 & 0 & 0.10 & 70.25778 & 0.000 & 0.00 & \textbf{ 70.25778} & 0.000 & 1 & 2 \\ 
  6 & 200 & 5 & 408 & 628 & 0.08 &  41.00533 & \textbf{ 41.00533} & 0.000 & 0 & 0.08 & 41.00533 & 0.000 & 0.00 & \textbf{ 41.00533} & 0.000 & 0 & 1 \\ 
  7 & 200 & 10 & 414 & 906 & 0.25 &  67.24628 & \textbf{ 67.24628} & 0.000 & 0 & 0.25 & 67.24628 & 0.000 & 0.01 &  67.13884 & 0.160 & 0 & 1 \\ 
  8 & 200 & 20 & 360 & 658 & 0.14 &  92.84587 & \textbf{ 92.84587} & 0.000 & 0 & 0.14 & 92.84587 & 0.000 & 0.00 & \textbf{ 92.84587} & 0.000 & 0 & 1 \\ 
  9 & 200 & 40 & 418 & 686 & 0.7 & 139.58466 & \textbf{139.58466} & 0.000 & 0 & 0.70 & 139.58466 & 0.000 & 0.04 & 139.26263 & 0.231 & 1 & 2 \\ 
  10 & 200 & 67 & 526 & 1372 & 1.77 & 182.01744 & \textbf{182.01744} & 0.000 & 1 & 1.77 & 182.01744 & 0.000 & 0.05 & 181.66281 & 0.195 & 3 & 2 \\ 
  11 & 300 & 5 & 1004 & 3968 & 0.26 &  94.69985 & \textbf{ 94.69985} & 0.000 & 0 & 0.26 & 94.69985 & 0.000 & 0.01 & \textbf{ 94.69985} & 0.000 & 0 & 1 \\ 
  12 & 300 & 10 & 814 & 2756 & 0.55 & 119.39182 & \textbf{119.39182} & 0.000 & 0 & 0.55 & 119.39182 & 0.000 & 0.02 & \textbf{119.39182} & 0.000 & 0 & 1 \\ 
  13 & 300 & 30 & 902 & 3342 & 3.01 & 196.56208 & \textbf{196.56208} & 0.000 & 3 & 2.99 & 196.56635 & 0.002 & 0.08 & 196.46042 & 0.052 & 0 & 1 \\ 
  14 & 300 & 60 & 804 & 2722 & 5.18 & 243.46959 & \textbf{243.46959} & 0.000 & 19 & 4.47 & 243.58466 & 0.047 & 0.20 & 242.57910 & 0.367 & 1 & 2 \\ 
  15 & 300 & 100 & 910 & 3470 & 9.65 & 281.15156 & \textbf{281.15156} & 0.000 & 6 & 9.10 & 281.16753 & 0.006 & 0.22 & 280.41637 & 0.262 & 4 & 2 \\ 
  16 & 400 & 5 & 1974 & 13784 & 1.88 & 173.70516 & \textbf{173.70516} & 0.000 & 5 & 1.79 & 173.93763 & 0.134 & 0.02 & \textbf{173.70516} & 0.000 & 0 & 1 \\ 
  17 & 400 & 10 & 1738 & 11266 & 5.93 & 213.34848 & \textbf{213.34848} & 0.000 & 21 & 4.02 & 214.36070 & 0.474 & 0.05 & \textbf{213.34848} & 0.000 & 0 & 1 \\ 
  18 & 400 & 40 & 1430 & 7758 & 18.48 & 300.60222 & \textbf{300.60222} & 0.000 & 25 & 15.00 & 300.85103 & 0.083 & 0.21 & 299.69177 & 0.304 & 1 & 2 \\ 
  19 & 400 & 80 & 1644 & 9618 & 32.75 & 368.98454 & \textbf{368.98454} & 0.000 & 40 & 26.47 & 369.05720 & 0.020 & 1.05 & 368.73928 & 0.067 & 3 & 2 \\ 
  20 & 400 & 133 & 1462 & 8014 & 54.25 & 394.49127 & \textbf{394.49127} & 0.000 & 25 & 48.42 & 394.51863 & 0.007 & 2.39 & 394.18127 & 0.079 & 3 & 2 \\ 
  21 & 500 & 5 & 2918 & 27626 & 2.14 & 240.85597 & \textbf{240.85597} & 0.000 & 1 & 2.14 & 240.85597 & 0.000 & 0.01 & \textbf{240.85597} & 0.000 & 0 & 1 \\ 
  22 & 500 & 10 & 2496 & 21494 & 23.61 & 270.50910 & \textbf{270.50910} & 0.000 & 52 & 8.54 & 273.78255 & 1.210 & 0.23 & 269.90640 & 0.223 & 0 & 1 \\ 
  23 & 500 & 50 & 2772 & 24348 & \textbf{TL} & 431.11067 & 430.59969 & 0.119 & 680 & 55.02 & 432.50301 & 0.442 & 0.86 & 429.34958 & 0.291 & 2 & 2 \\ 
  24 & 500 & 100 & 2810 & 23810 & 102.06 & 483.42587 & \textbf{483.42587} & 0.000 & 94 & 62.16 & 483.60746 & 0.038 & 1.20 & 483.31644 & 0.023 & 2 & 2 \\ 
  25 & 500 & 167 & 3078 & 29866 & \textbf{TL} & 499.06847 & 499.06780 & 0.000 & 480 & 293.78 & 499.07072 & 0.001 & 3.96 & 498.64504 & 0.085 & 5 & 2 \\ 
  26 & 600 & 5 & 5086 & 70238 & 111.75 & 343.15861 & \textbf{343.15861} & 0.000 & 94 & 12.16 & 350.42676 & 2.118 & 0.04 & 342.80100 & 0.104 & 0 & 1 \\ 
  27 & 600 & 10 & 5420 & 75078 & \textbf{TL} & 418.49400 & 417.26694 & 0.294 & 296 & 30.06 & 422.97740 & 1.369 & 0.09 & 416.29616 & 0.233 & 0 & 1 \\ 
  28 & 600 & 60 & 5338 & 78458 & \textbf{TL} & 560.06552 & 558.94592 & 0.200 & 139 & 161.65 & 560.47718 & 0.274 & 1.42 & 558.38867 & 0.100 & 3 & 2 \\ 
  29 & 600 & 120 & 4982 & 67654 & \textbf{TL} & 592.56902 & 592.41242 & 0.026 & 91 & 273.00 & 592.71083 & 0.050 & 3.75 & 591.87448 & 0.091 & 2 & 2 \\ 
  30 & 600 & 200 & 4284 & 48514 & 21.5 & 600.00000 & \textbf{600.00000} & 0.000 & 0 & 21.50 & 600.00000 & 0.000 & 2.53 & \textbf{600.00000} & 0.000 & 3 & 2 \\ 
  31 & 700 & 5 & 8624 & 152102 & \textbf{TL} & 475.09728 & 471.24301 & 0.818 & 111 & 33.75 & 482.82600 & 2.458 & 0.04 & 470.24484 & 0.212 & 0 & 1 \\ 
  32 & 700 & 10 & 7294 & 123992 & \textbf{TL} & 516.14680 & 512.62608 & 0.687 & 135 & 42.78 & 519.29985 & 1.302 & 0.32 & 512.62607 & 0.000 & 1 & 2 \\ 
  33 & 700 & 70 & 7110 & 121930 & \textbf{TL} & 674.32202 & 673.19323 & 0.168 & 20 & 372.46 & 674.42293 & 0.183 & 2.23 & 672.06564 & 0.168 & 3 & 2 \\ 
  34 & 700 & 140 & 8044 & 141988 & \textbf{TL} & 699.35981 & 699.35796 & 0.000 & 0 & 616.80 & 699.35981 & 0.000 & 6.50 & 699.19755 & 0.023 & 3 & 2 \\ 
  35 & 800 & 5 & 14838 & 312066 & \textbf{TL} & 623.45089 & 611.90493 & 1.887 & 30 & 69.20 & 624.34885 & 2.034 & 0.07 & 611.72016 & 0.030 & 0 & 1 \\ 
  36 & 800 & 10 & 10904 & 224412 & \textbf{TL} & 635.07708 & 621.43791 & 2.195 & 32 & 80.34 & 635.54000 & 2.269 & 0.19 & 620.29056 & 0.185 & 0 & 1 \\ 
  37 & 800 & 80 & 9072 & 186296 & \textbf{TL} & 778.47803 & 777.14445 & 0.172 & 0 & 609.69 & 778.47803 & 0.172 & 5.96 & 776.88145 & 0.034 & 2 & 2 \\ 
  38 & 900 & 5 & 21526 & 466312 & \textbf{TL} & 736.92875 & 723.50305 & 1.856 & 13 & 111.86 & 737.79255 & 1.975 & 0.13 & 723.50304 & 0.000 & 0 & 1 \\ 
  39 & 900 & 10 & 21248 & 471774 & \textbf{TL} & 796.72085 & 787.65539 & 1.151 & 12 & 159.43 & 797.01301 & 1.188 & 0.20 & 787.65527 & 0.000 & 0 & 1 \\ 
  40 & 900 & 90 & 14636 & 345298 & \textbf{TL} & 890.67134 & 889.55816 & 0.125 & 0 & 606.54 & 890.67134 & 0.125 & 6.92 & 889.11200 & 0.050 & 2 & 2 \\ 
   \bottomrule
\end{tabular}
\endgroup
\end{table}
\end{landscape}

\section{Conclusions and future work \label{sec:con}}

In this paper we studied the recently introduced
\emph{multiple gradual cover location problem} (\MGCLP, see~\cite[][]{berman2018multiple}).
The \MGCLP\ addresses, simultaneously, two issues that have been identified 
as relevant in practical facility location applications:
gradual coverage and potential co-location of facilities.
We presented four different mixed integer programming formulations for the \MGCLP,
all of them exploiting the submodularity of the objective function.
Furthermore, we designed and implemented a branch-and-cut framework based one these formulations.
Our framework is further enhanced by additional cut separation strategies,
starting and primal heuristics and initialization procedures.

From an algorithmic perspective, 
the computational results show that our approach
allows to effectively address different sets of instances. We provide optimal solution values for 13 instances from literature, where the optimal solution was not known, and additionally provide improved solution values for seven instances.
%
We also analyzed the dependence of the solution-structure on instance-characteristics.
The reported results show that the \MGCLP\ possesses a great capability for allowing decision makers to design their facility deployment strategies according
to different coverage capacities or customer preferences.
Interestingly, the results show that although multiple location
of facilities indeed occurs (there are even cases where seven facilities are 
located at the same location), the magnitude of the co-location is rather small.

A wide class of strategical and tactical decisions
in many operations research settings correspond to facility location deployment. 
Incorporating modeling features such as partial and joint coverage
along with co-location of facilities, establishes a new path
for closing the gap between academic optimization tools
and real-world location problems.
For future work, it would be interesting to study
how different partial coverage functions can be included
into the proposed models, or how the different degrees
of uncertainty (e.g., location of customers) can be incorporated
within the proposed modeling and algorithmic frameworks.

\section*{Acknowledgements}

E.A.-M. acknowledges  the  support  of  the  Chilean  Council  of  Scientific  and  Technological  Research,  CONICYT,  through  the  grant  FONDECYT  N.1180670  and  through  the  Complex  Engineering Systems Institute 
(ICM-FIC:P-05-004-F, CONICYT:FB0816). The research of M.S. was supported by the Austrian Research Fund (FWF, Project P 26755-N19).

\bibliographystyle{plainnat}
\bibliography{RRUFLP}

\end{document}